\documentclass[11pt]{article}
\usepackage{amsmath,amsthm,amsfonts,latexsym,amsopn,verbatim,amscd,amssymb}
\usepackage{adjustbox}
\usepackage{lipsum}
\usepackage{cleveref}
\usepackage{enumerate}
\usepackage[left=2cm, right=2cm, top=2cm]{geometry}

\newtheorem{theorem}{Theorem}[section]
\newtheorem{Assumption}{Assumption}[section]

\newtheorem{Example}{Example}[section]

\newtheorem{lemma}{Lemma}[section]

\usepackage{tabularx}
\usepackage{amsmath,amsfonts,amssymb}
\usepackage{graphics, graphicx}
\usepackage{enumitem}
\usepackage{subcaption}
\usepackage{multirow}
\usepackage{authblk}
\usepackage{enumitem,amssymb}
\usepackage{graphicx,float}
\newlist{todolist}{itemize}{2}
\setlist[todolist]{label=$\square$}
\begin{document}
	\title{PARAMETER UNIFORM NUMERICAL METHOD FOR SINGULARLY PERTURBED TWO PARAMETER PARABOLIC PROBLEM WITH DISCONTINUOUS CONVECTION COEFFICIENT AND SOURCE TERM}
	\date{}
\author{Nirmali Roy$^1$, Anuradha Jha$^{2^*}$}
\date{%
    $^{1,2}$Indian Institute of Information Technology Guwahati, Bongora, India,781015\\
   *Corresponding author(s). E-mails: {anuradha@iiitg.ac.in}
}
	\maketitle
	\section*{Abstract}In this article, we have considered a time-dependent two-parameter singularly perturbed parabolic problem with discontinuous convection coefficient and source term. The problem contains the parameters $\epsilon$  and $\mu$ multiplying the diffusion and convection coefficients, respectively. A boundary layer develops on both sides of the boundaries as a result of these parameters.  An interior layer forms near the point of discontinuity due to the discontinuity in the convection and source term. The width of the interior and boundary layers depends on the ratio of the perturbation parameters. We discuss the problem for ratio $\displaystyle\frac{\mu^2}{\epsilon}$. We used an upwind finite difference approach on a Shishkin-Bakhvalov mesh in the space and the Crank-Nicolson method in time on uniform mesh. At the point of discontinuity, a three-point formula was used. This method is uniformly convergent with second order in time and first order in space. Shishkin-Bakhvalov mesh provides first-order convergence; unlike the Shishkin mesh, where a logarithmic factor deteriorates the order of convergence. Some test examples are given to validate the results presented.\\
		
		\textbf{Keywords}: singularly perturbed, two-parameter, parabolic problem, discontinuous data,  boundary and interior layers, Crank-Nicolson method, Shishkin-Bakhvalov mesh


\section{Introduction}	
Here, we have considered the following two-parameter singularly perturbed parabolic boundary value problem(TPSPP-BVP) on the domain $\Omega=(0,1)\times (0,T]$:
\begin{equation}
\begin{aligned}
\label{twoparaparabolic}
&\mathcal{L}_{\epsilon,\mu} u(x,t)\equiv(\epsilon u_{xx}+\mu au_x-bu-cu_t)(x,t)=f(x,t),~~(x,t)\in(\Omega^{-} \cup \Omega^{+}),\\
&u(0,t)
=p(t),~ (0,t)\in\Gamma_l;\\
&u(1,t)=r(t),(1,t)\in\Gamma_r;\\
&u(x,0)=q(x), (x,0)\in\Gamma_b,
\end{aligned}
\end{equation}
where $0<\epsilon\ll1$ and $0<\mu\le1$ are two singular perturbation parameters and $\Omega^{-}=(0,d)\times(0,T],\Omega^{+}=(d,1)\times(0,T]$. The convection coefficient $a(x,t)$ and source term $f(x,t)$ are discontinuous at $(d,t)\in \Omega ~\forall ~t.$
	Also $a(x,t)\le-\alpha_1<0, (x,t)\in\Omega^-$ and $a(x,t)\ge\alpha_2>0, (x,t)\in\Omega^+$, $\alpha_1, \alpha_2$ are positive constants. The functions $a(x,t)$ and $ f(x,t)$  are sufficiently smooth functions on $(\Omega^- \cup \Omega^+)$.  In addition, we assume the jumps of $a(x,t)$ and $f(x,t)$ at $(d,t)$ satisfy $|[a](d,t)|<C,~~|[f](d,t)|<C$, where the jump of $\omega$ at $(d,t)$ is defined as  $[\omega](d,t)=\omega(d+,t)-\omega(d-,t)$. 
	The coefficients $ b(x, t)$ and $c(x, t)$ are assumed to be sufficiently smooth functions on $\Omega$ such that $b(x,t)\ge\beta>0$ and $ c(x,t)>\eta>0$. The boundary data $p(t), r(t)$ and initial data $q(x)$ are sufficiently smooth on the domain and satisfy the compatibility condition. Let
	$\Gamma^{-}=(0,d),\Gamma^{+}=(d,1), \Gamma=(0,1),$  $\Gamma_l=\{(0,t)|0\le t\le T\},\Gamma_r=\{(1,t)|0\le t\le T\}$, $\Gamma_b=\{(x,0)|0\le x\le1\}$ and $\Gamma_{c}=\Gamma_{l}\cup\Gamma_{b}\cup\Gamma_{r}$.
 Under the above assumptions, the problem \eqref{twoparaparabolic} has a unique continuous solution in the domain $\bar{\Omega}$. 
	
	The solution to problem \eqref{twoparaparabolic} contains interior layers at the points of discontinuity $ (d,t),\forall t \in (0,T]$ because of the discontinuity in the convection coefficient and source term. Additionally, the solution displays boundary layers at $\Gamma_{l}$ and $\Gamma_{r}$ as a result of the presence of perturbation parameters $\epsilon$ and $\mu$.
	
     Two-parameter singularly perturbed problems was first studied asymptotically by O'Malley in {\cite{OM1,OM,OM2}}. He observed that the solution to these problems depends on the perturbation parameters $\epsilon$ and $\mu$ as well as on their ratio. So, we discuss the solution of Eq. \eqref{twoparaparabolic} under the following cases:\\
	\textbf{Case (i)}: $\sqrt{\alpha}\mu \le \sqrt{\rho\epsilon}$ , where $\rho=\min\limits_{x\in\bar{\Omega}}\bigg\{\frac{|b(x,t)|}{|a(x,t)|}\bigg\}$, the boundary layers of equal width of $\mathcal{O}(\sqrt{\epsilon})$ appear near the boundaries.\\
	\textbf{Case (ii)}: $\sqrt{\alpha}\mu > \sqrt{\rho\epsilon}$, the solution has boundary layers at both the boundaries of different widths.
	
	In recent years, several authors have developed numerical methods for two-parameter singularly perturbed parabolic problems with smooth data \cite{BDD,GKD, KS,MD,M,ER1,ZAM}. In the case of non-smooth data, the study is very limited. M. Chandru et al. in \cite{MC1} considered the problem \eqref{twoparaparabolic} and proved that the upwind scheme on space on Shishkin mesh and backward Euler scheme on time is almost first-order accurate. D. Kumar et al. in \cite{DP} gave a numerical method with parameter uniform convergence of order two in time and almost order one in space for the problem of the same type. They used the Crank-Nicolson method in time on a uniform mesh and the upwind method in space on a Shishkin mesh.
	Some work on one parameter parabolic singularly perturbed problem with discontinuous data includes \cite{TA,ES}.
	
	 In this paper, we have used Crank-Nicolson scheme \cite{CN1} on time on a uniform mesh and upwind scheme on an appropriately defined  Shishkin-Bakhvalov mesh \cite{TL1,TL2} in space for case (i). At the point of discontinuity, we used a three-point scheme to resolve it. In the Shishkin-Bakhvalov mesh, we choose transition point as in Shishkin mesh and use graded mesh (as in Bakhvalov \cite{B}) in the layer region. In the outer region a uniform mesh is used. We are able to achieve first-order convergence in space due to Shishkin-Bakhvalov mesh and second-order in time due to Crank-Nicolson method.
	 
	 The article is organised in the order listed below: The a\textit{priori} bounds on the solution and its derivatives are given in section 2. The decomposition of the continuous solution into regular and singular components and bounds on the derivatives of regular and singular components are also discussed here.
     In section 3, numerical method and mesh construction is discussed. In section 4, parameter uniform error estimates are established for case $\sqrt{\alpha}\mu \le \sqrt{\rho\epsilon}$. The numerical examples in section 5 support the theoretical results given in previous section. A summary of the work is presented in section 6.
	 
	{\bf{Notation}}:  The norm used is the maximum norm given by
	$$\|u\|_{\Omega}=\max_{(x,t)\in \Omega}|u(x,t)|.$$
Throughout the paper, $C, C_{1}, C_{2}$ will be denoted as a generic positive constant that is independent of perturbation parameters $\epsilon, \mu$ and mesh size.


\section{A \textit{Priori} Derivatives Bounds of the Solution}
\label{2}
\begin{theorem}\cite{DP}
	Suppose that a function $z(x,t)\in C^{0}(\bar{\Omega})\cap C^{2}(\Omega^{-}\cup \Omega^{+})$ satisfies $z(x,t)\ge 0, ~\forall(x,t)\in \Gamma_{c},~\mathcal{L}z(x,t)\le 0, \forall(x,t)\in(\Omega^{-}\cup\Omega^{+})$ and $[z_{x}](d,t) \le 0,t>0$ then $z(x,t)\geq 0 \; \forall ~(x,t)\in \bar{\Omega}$.
\end{theorem}
\begin{theorem}\cite{DP}
	The bounds on $u(x,t)$ is given by 
	$$\|u\|_{\bar{\Omega}}\le \|u\|_{\Gamma_{c}}+\frac{\|f\|_{\Omega^{-}\cup\Omega^{+}}}{\beta}.$$
\end{theorem}
The solution $u(x,t)$ of Eq.\eqref{twoparaparabolic}  is divided as in \cite{DP} into layers and regular components as $u(x,t)=v(x,t)+w_{l}(x,t)+w_{r}(x,t)$.
The regular component $v(x,t) $ satisfies the following equation:
\begin{equation}
\begin{aligned}
\label{v}
&\mathcal{L}v(x,t)=f(x,t),~~(x,t)\in\Omega^{-}\cup\Omega^{+},\\
&v(0,t)=u(0,t),v(1,t)=u(1,t),~\forall~t\in(0,T],\\
&v(x,0)=u(x,0), ~\forall~x\in \Gamma^{-}\cup \Gamma^{+},\\
&v(d-,t),v(d+,t) ~\text{are chosen suitably}~ \forall~t\in(0,T].
\end{aligned}
\end{equation}
$v(x,t)$ can further be decomposed as
$$v(x,t)=\left\{
\begin{array}{ll}
\displaystyle v^{-}(x,t), & \hbox{ $(x,t)\in \Omega^-,$ }\\
v^{+}(x,t), & \hbox{ $(x,t)\in \Omega^+,$ }
\end{array}
\right.$$
where $v^{-}$ and $v^{+}$ are the left and right regular components respectively.

The singular components $w_{l}(x,t)$ and $w_{r}(x,t)$ are the solutions of
\begin{equation}
\begin{aligned}
\label{lw}
&\mathcal{L}w_{l}(x,t)=0,~~(x,t)\in \Omega^{-}\cup\Omega^{+},\\
&w_{l}(0,t)=u(0,t)-v(0,t)-w_{r}(0,t),~\forall~ t\in(0,T],\\
&w_{l}(1,t) \text{are chosen suitably}~\forall~ t\in(0,T],\\
&w_{l}(x,0)=u(x,0),~\forall~x\in \Gamma^{-}\cup \Gamma^{+}, \\
\end{aligned}
\end{equation}
and 
\begin{equation}
\begin{aligned}
\label{rw}
&\mathcal{L}w_{r}(x,t)=0,~~(x,t)\in \Omega^{-}\cup\Omega^{+},\\
&w_{r}(1,t)=u(1,t)-v(1,t)-w_{l}(1,t),~\forall~t\in(0,T], \\
& w_{r}(0,t) ~\text{are chosen suitably} ~\forall~ t\in(0,T],\\
&w_{r}(x,0)=u(x,0),~\forall~x\in \Gamma^{-}\cup \Gamma^{+},
\end{aligned}
\end{equation}
respectively. Also 
\[[w_{r}](d,t)=-([v]+[w_{l}])(d,t),~ [(w_{r})_{x}](d,t)=-([v_{x}]+[(w_{l})_{x}])(d,t),~\forall~ t\in(0,T].\]
Further, the singular components $w_{l}(x,t)$ and $w_{r}(x,t)$ are decomposed as
$$w_{l}(x,t)=\left\{
\begin{array}{ll}
\displaystyle w_{l}^{-}(x,t), & \hbox{ $(x,t)\in \Omega^-,$ }\\
w_{l}^{+}(x,t), & \hbox{ $(x,t)\in \Omega^+,$ }
\end{array}
\right.
w_{r}(x,t)=\left\{
\begin{array}{ll}
\displaystyle w_{r}^{-}(x,t), & \hbox{ $(x,t)\in \Omega^-,$ }\\
w_{r}^{+}(x,t), & \hbox{ $(x,t)\in \Omega^+,$ }
\end{array}
\right.$$
where $w_{l}^{-}$, $w_{r}^{-}$ are left components and  $w_{l}^{+}$, $w_{r}^{+}$ are right components of the singular component respectively.
Hence, the unique solution $u(x,t)$ to the problem Eq.\eqref{twoparaparabolic} is written as
$$u(x,t)=\left\{
\begin{array}{ll}
\displaystyle (v^{-}+w_{l}^{-}+w_{r}^{-})(x,t),\hspace{4.9cm} (x,t)\in\Omega^{-},\\
(v^{-}+w_{l}^{-}+w_{r}^{-})(d-,t)=(v^{+}+w_{l}^{+}+w_{r}^{+})(d+,t), ~~ (x,t)=(d,t),\forall~ t\in(0,T],\\
(v^{+}+w_{l}^{+}+w_{r}^{+})(x,t),\hspace{4.9cm}(x,t)\in\Omega^{+}.
\end{array}
\right.$$
\begin{theorem}
	\label{wlb}
The left layer and right layer components $w_{l}(x,t)$ and $w_{r}(x,t)$ satisfy the following inequalities for case $\sqrt{\alpha}\mu\le\sqrt{\rho \epsilon}$
	$$\|w_{l}(x,t)\|_{\Omega^{-}\cup
		\Omega^{+}} \le \left\{
	\begin{array}{ll}
	\displaystyle Ce^{-\theta_{1}x}, & \hbox{ $x\in \Omega^{-},$ }\\
	Ce^{-\theta_{2}(x-d)}, & \hbox{ $x\in \Omega^{+}$,}
	\end{array}
	\right.$$
	$$\|w_{r}(x,t)\|_{\Omega^{-}\cup
		\Omega^{+}} \le \left\{
	\begin{array}{ll}
	\displaystyle Ce^{-\theta_{2}(d-x)}, & \hbox{ $x\in \Omega^{-}$, }\\
	Ce^{-\theta_{1}(1-x)}, & \hbox{ $x\in \Omega^{+}$,}
	\end{array}
	\right.$$
	where
\begin{equation}
\label{theta}
	\theta_{1}=\frac{\sqrt{\rho \alpha}}{\sqrt{\epsilon}},~~\theta_{2}=\frac{\sqrt{\rho \alpha}}{2\sqrt{\epsilon}}.
\end{equation}
	For the case $\sqrt{\alpha}\mu>\sqrt{\rho \epsilon}$, the bounds of left layer and right layer components satisfy the following inequalities
	$$\|w_{l}(x,t)\|_{\Omega^{-}\cup
		\Omega^{+}} \le \left\{
	\begin{array}{ll}
	\displaystyle Ce^{-\theta_{1}x}, & \hbox{ $x\in \Omega^{-},$ }\\
	Ce^{-\theta_{2}(x-d)}, & \hbox{ $x\in \Omega^{+}$,}
	\end{array}
	\right.$$
	$$\|w_{r}(x,t)\|_{\Omega^{-}\cup
		\Omega^{+}} \le \left\{
	\begin{array}{ll}
	\displaystyle Ce^{-\theta_{2}(d-x)}, & \hbox{ $x\in \Omega^{-}$, }\\
	Ce^{-\theta_{1}(1-x)}, & \hbox{ $x\in \Omega^{+}$,}
	\end{array}
	\right.$$
\end{theorem}
where
\begin{equation}
\label{theta}
	\theta_{1}=\frac{\alpha\mu}{\epsilon},~~\theta_{2}=\frac{\gamma}{2\mu}.
\end{equation}
\begin{proof}
	The proof follows from the technique given in \cite{ES, ER1}.
\end{proof}
\begin{theorem}
	\label{wwlbd}
	Let  $\sqrt{\alpha}\mu\le\sqrt{\rho \epsilon}$, the singular component $w_{l}(x,t)$ and $w_{r}(x,t)$ satisfies the following bounds 	for $k=0,1,2,3:$\\
	\[\bigg\|\frac{d^{k}v^{-}(x,t)}{dx^{k}}\bigg\|_{\Omega^{-}} \le C\bigg(1+\epsilon^{-\frac{(k-3)}{2}}\bigg),~~
	\bigg\|\frac{d^{k}v^{+}(x,t)}{dx_{k}}\bigg\|_{\Omega^{+}} \le C\bigg(1+\epsilon^{-\frac{(k-3)}{2}}\bigg),\]
 	\[\bigg\|\frac{d^{k}w_{l}^{-}(x,t)}{dx^{k}}\bigg\|_{\Omega^{-}} \le C\epsilon^{-\frac{k}{2}},~~~~
	\bigg\|\frac{d^{k}w_{l}^{+}(x,t)}{dx^{k}}\bigg\|_{\Omega^{+}} \le C\epsilon^{-\frac{k}{2}},\]
	\[\bigg\|\frac{d^{k}w_{r}^{-}(x,t)}{dx^{k}}\bigg\|_{\Omega^{-}} \le C\epsilon^{-\frac{k}{2}},~~~~
	\bigg\|\frac{d^{k}w_{r}^{+}(x,t)}{dx^{k}}\bigg\|_{\Omega^{+}} \le C\epsilon^{-\frac{k}{2}}.\]
	
	For $\sqrt{\alpha}\mu>\sqrt{\rho \epsilon}$, the singular component $w_{l}(x,t)$ and $w_{r}(x,t)$ satisfies the following bounds
	\[\bigg\|\frac{d^{k}v^{-}(x,t)}{dx^{k}}\bigg\|_{\Omega^{-}} \le C\bigg(1+\bigg(\frac{\epsilon}{\mu}\bigg)^{3-k}\bigg),~~
	\bigg\|\frac{d^{k}v^{+}(x,t)}{dx_{k}}\bigg\|_{\Omega^{+}} \le C\bigg(1+\bigg(\frac{\epsilon}{\mu}\bigg)^{3-k}\bigg),\]
 	\[\bigg\|\frac{d^{k}w_{l}^{-}(x,t)}{dx^{k}}\bigg\|_{\Omega^{-}} \le C\bigg(\frac{\mu}{\epsilon}\bigg)^{k},~~~~
	\bigg\|\frac{d^{k}w_{l}^{+}(x,t)}{dx^{k}}\bigg\|_{\Omega^{+}} \le C\mu^{-k},\]
	\[\bigg\|\frac{d^{k}w_{r}^{-}(x,t)}{dx^{k}}\bigg\|_{\Omega^{-}} \le C\mu^{-k},~~~~
	\bigg\|\frac{d^{k}w_{r}^{+}(x,t)}{dx^{k}}\bigg\|_{\Omega^{+}} \le C\bigg(\frac{\mu}{\epsilon}\bigg)^{k}.\]
\end{theorem}
\begin{proof}
Using the techniques given in \cite{ES, ER1}, the proof follow.
\end{proof}

\section{Difference scheme}
\label{3}
We use the horizontal MOL to discretize the time variable 
using the Crank-Nicolson method \cite{CN1},
with constant step-size $\Delta t$, while keeping the variable $x$ continuous.
 For a fixed time $T$, the interval $[0,T]$ is partitioned uniformly as $\Lambda^{M}=\{t_{j}=j\Delta t:j=0,1,\ldots,M, \Delta t=\frac{T}{M}\}$. The semi-discretization yields the following system of linear ordinary differential equations:
\begin{align*}
&\epsilon U^{j+\frac{1}{2}}_{xx}(x)+\mu a^{j+\frac{1}{2}}(x)U^{j+\frac{1}{2}}_{x}(x)-b^{j+\frac{1}{2}}(x)U^{j+\frac{1}{2}}(x)=f^{j+\frac{1}{2}}(x)+\frac{U^{j+1}(x)-U^{j}(x)}{\Delta t},\\
&x\in  (\Omega^{-}\cup\Omega^{+}),~ 0\le j\le M-1,\\
&U^{j+1}(0)=u(0,t_{j+1}),~~~U^{j+1}(1)=u(1,t_{j+1}),~~0\le j\le M-1,\\
&U^{0}(x)=u(x,0),~~x\in\Gamma,
\end{align*}
where $U^{j+1}(x)$ is the approximation of $u(x,t_{j+1})$ of Eq. \eqref{twoparaparabolic} at $(j+1)$-th time level and $U^{j+\frac{1}{2}}=\frac{U^{j+1}(x)+U^{j}(x)}{2} $.
After simplification, we obtain
\begin{equation}
\label{semi-dis}
\left\{
\begin{array}{ll}
\displaystyle \mathcal{\tilde{L}}U^{j+1}(x)=g(x,t_{j+1}), & \hbox{ $x\in  (\Gamma^{-}\cup\Gamma^{+}),~ 0\le j\le M-1$, }\\
U^{j+1}(0)=u(0,t_{j+1}),U^{j+1}(1)=u(1,t_{j+1}), & \hbox{ $0\le j\le M-1$,}\\
U^{0}(x)=u(x,0), & \hbox{ $x\in\Gamma,$ }
\end{array}
\right.
\end{equation}
where the operator $\mathcal{\tilde{L}}$ is defined as\\
$$\mathcal{\tilde{L}}\equiv \epsilon \frac{\delta^{2}}{\delta x^{2}}+\mu a^{j+\frac{1}{2}}\frac{\delta}{\delta x}-c^{j+\frac{1}{2}}I,$$\\
and $g(x,t_{j+1})=2f^{j+\frac{1}{2}}(x)-\epsilon U^{j}_{xx}(x)-\mu a^{j+\frac{1}{2}}(x)U_{x}^{j}(x)+d^{j+\frac{1}{2}}(x)U^{j}(x),$\\
$c^{j+\frac{1}{2}}(x)=b^{j+\frac{1}{2}}(x)+\frac{2}{\Delta t},$\\
$d^{j+\frac{1}{2}}(x)=b^{j+\frac{1}{2}}(x)-\frac{2}{\Delta t}.$

The error in temporal semi-discretization is defined by 
$e^{j+1}= u(x,t_{j+1})- \hat{U}^{j+1}(x)$ where $u(x,t_{j+1})$ is the solution of  Eq. \eqref{twoparaparabolic}. $ \hat{U}^{j+1}(x)$  is the solution of semi-discrete equation \eqref{semi-dis}, when $u(x,t_{j})$ is taken instead of $U^{j}$ to find solution at $(x,t_{j+1})$.
\begin{theorem}
	The local truncation error $T_{j+1}=\mathcal{\tilde{L}}(e^{j+1})$ satisfies 
	$$\|T_{j+1}\|\le C(\Delta t)^{3},~~~0\le j\le M-1.$$
\end{theorem}
\begin{proof}
The proof follows from the technique given in \cite{KTK}.
\end{proof}
\begin{theorem}
	The global  error $E^{j+1}=u(x,t_{j+1})-U^{j+1}(x)$ is estimated as 
	$$\|E^{j+1}\|\le C(\Delta t)^{2},~~~0\le j\le M-1.$$
	Here $U^{j+1}(x)$ is the solution of Eq.\eqref{semi-dis}.
\end{theorem}
\begin{proof}
	The proof follows from the technique given in \cite{KTK}.
\end{proof}
We will now define the fully discretized scheme. In spatial direction, we use a non-uniform mesh that is graded in the layer region and uniform in the outer region.
The semi-discrete problem in \eqref{semi-dis} is discretized using the upwind finite difference method on an appropriately defined Shishkin-Bakhvalov mesh on space.  Let the interior points of the spatial mesh are denoted by $\Gamma^{N}=\{x_{i}:1\le i \le \frac{N}{2}-1\}\cup\{x_{i}:\frac{N}{2}+1\le i \le N-1\}.$
The $\bar{\Gamma}^{N}=\{x_{i}\}_{0}^{N}\cup\{d\}$ denote the mesh points with $x_{0}=0,x_{N}=1$ and the point of discontinuity at point $x_{\frac{N}{2}}=d.$ We also introduce the notation $\Gamma^{N-}=\{x_{i}\}_{0}^{\frac{N}{2}-1}, \Gamma^{N+}=\{x_{i}\}^{N-1}_{\frac{N}{2}+1}$, $\Omega^{N-}=\Gamma^{N-}\times\Lambda^{M}$, $\Omega^{N+}=\Gamma^{N+}\times\Lambda^{M}$and $\bar{\Omega}^{N,M}=\bar{\Gamma}^{N}\times\Lambda^{M}$.
The domain $[0,1]$ is subdivided into six sub-intervals as
\[\bar{\Gamma}=[0,\tau_{1}]\cup[\tau_{1},d-\tau_{2}]\cup[d-\tau_{2},d]\cup[d,d+\tau_{3}]\cup[d+\tau_{3},1-\tau_{4}]\cup[1-\tau_{4},1].\]
The transition points are defined as done for Shishkin mesh:
$$\tau_{1}=\frac{4}{\theta_{1}}\ln N,~~~\tau_{2}=\frac{4}{\theta_{2}}\ln N,$$ 
$$\tau_{3}=\frac{4}{\theta_{2}}\ln N,~~~\tau_{4}=\frac{4}{\theta_{1}}\ln N.$$
where $\theta_1$ and $\theta_2$ are the same as in \eqref{theta}.

The interval $[0,\tau_{1}]$ is subdivided into $\frac{N}{8}$ sub-intervals by inverting the function  $e^{-\theta_{1}x}$ linearly in it, so for $i=0,\ldots,\frac{N}{8}$,
$$e^{-\theta_{1}x_{i}/8}=Ai+B$$ 
such that $x_{0}=0, x_{\frac{N}{8}}=\tau_{1}$. Thus, we obtain 
$$x_{i}=-\frac{8}{\theta_{1}}\log\bigg(1+\frac{8i}{N}\bigg(\frac{1}{\sqrt{N}}-1\bigg)\bigg),~~0\le i \le \frac{N}{8}.$$

In interval $[d-\tau_{2},d]$, we invert the function  $e^{-\theta_{2}(x-d)}$ linearly to obtain $\frac{N}{8}+1$ mesh points. Thus, we obtain 
$$x_{i}=d+\frac{8}{\theta_{2}}\log\bigg(\frac{8i}{N}\bigg(1-\frac{1}{\sqrt{N}}\bigg)+\frac{4}{\sqrt{N}}-3\bigg),~~\frac{3N}{8}\le i \le \frac{N}{2}.$$
where $x_{\frac{3N}{8}}=d-\tau_{2}$ and $x_{\frac{N}{2}}=d.$
Similarly by inverting the function $e^{-\theta_{2}(d-x)}$ linearly in the interval  $[d,d+\tau_{3}]$, we obtain the following  $\frac{N}{8}+1$ mesh points for $\frac{N}{2}\le i \le \frac{5N}{8}$:
$$x_{i}=d-\frac{8}{\theta_{2}}\log\bigg(\frac{8i}{N}\bigg(\frac{1}{\sqrt{N}}-1\bigg)+5-\frac{4}{\sqrt{N}}\bigg),~~\frac{N}{2}\le i \le \frac{5N}{8}$$
where $x_{\frac{N}{2}}=d, x_{\frac{5N}{8}}=d+\tau_{3}$.
Also, the interval $[1-\tau_{4},1]$ can be subdivided into $\frac{N}{8}$ sub-intervals by inverting the function  $e^{-\theta_{1}(1-x)}$ linearly in it, 
so, we have 
$$x_{i}=1+\frac{8}{\theta_{1}}\log\bigg(\frac{8i}{N}\bigg(1-\frac{1}{\sqrt{N}}\bigg)+\frac{8}{\sqrt{N}}-7\bigg),~~\frac{7N}{8}\le i \le N.$$

A uniform mesh consists of $\big(\frac{N}{4}+1\big)$ mesh points is employed between intervals $[\tau_{1},d-\tau_{2}]$ and $[d+\tau_{3},1-\tau_{4}]$.

The mesh points in space are given by
$$x_{i}=\left\{
\begin{array}{ll}
\vspace{0.2cm}
\displaystyle -\frac{8}{\theta_{1}}\log\bigg(1+\frac{8i}{N}\bigg(\frac{1}{\sqrt{N}}-1\bigg)\bigg), & \hbox{ $0\le i \le \frac{N}{8},$ }\\
\vspace{0.2cm}
\displaystyle\tau_{1}+\frac{(d-\tau_{1}-\tau_{2})\bigg(\frac{i}{N}-\frac{1}{8}\bigg)}{\frac{1}{4}}, & \hbox{ $\frac{N}{8}\le i \le \frac{3N}{8},$ }\\
\vspace{0.2cm}
\displaystyle d+\frac{8}{\theta_{2}}\log\bigg(\frac{8i}{N}\bigg(1-\frac{1}{\sqrt{N}}\bigg)+\frac{4}{\sqrt{N}}-3\bigg), & \hbox{ $\frac{3N}{8}\le i \le \frac{N}{2},$ }\\
\vspace{0.2cm}
\displaystyle d-\frac{8}{\theta_{2}}\log\bigg(\frac{8i}{N}\bigg(\frac{1}{\sqrt{N}}-1\bigg)+5-\frac{4}{\sqrt{N}}\bigg), & \hbox{ $\frac{N}{2}\le i \le \frac{5N}{8},$ }\\
\vspace{0.2cm}
\displaystyle d+\tau_{3}+\frac{(1-d-\tau_{3}-\tau_{4})\bigg(\frac{i}{N}-\frac{5}{8}\bigg)}{\frac{1}{4}}, & \hbox{ $\frac{5N}{8}\le i \le \frac{7N}{8},$ }\\
\displaystyle1+\frac{8}{\theta_{1}}\log\bigg(\frac{8i}{N}\bigg(1-\frac{1}{\sqrt{N}}\bigg)+\frac{8}{\sqrt{N}}-7\bigg), & \hbox{ $\frac{7N}{8}\le i \le N .$ }
\end{array}
\right.$$

In terms of the mesh generating function $\phi$, which maps a uniform mesh $\xi$ onto a layer adapted mesh in $x$ by $x=\phi(\xi)$. The mesh can be written as:
\[	x_i=\phi(\xi_i)=\left\{
\begin{array}{ll}
\displaystyle \frac{8}{\theta_{1}}\phi_{1}(\xi_i), & \hbox{ $ 0\le i \le \frac{N}{8},$ }\\
\displaystyle \tau_{1}+\frac{(d-\tau_{1}-\tau_{2})(\xi_i-\frac{1}{8})}{\frac{1}{4}}, & \hbox{ $\frac{N}{8}\le i \le \frac{3N}{8},$ }\\
\vspace{0.3cm}
\displaystyle d-\frac{8}{\theta_{2}}\phi_{2}(\xi_i), & \hbox{ $\frac{3N}{8}\le i \le \frac{N}{2},$ }\\
\displaystyle d+\frac{8}{\theta_{2}}\phi_{3}(\xi_i), & \hbox{ $\frac{N}{2}\le i \le \frac{5N}{8},$ }\\
\displaystyle d+\tau_{3}+\frac{(1-d-\tau_{3}-\tau_{4})(\xi_i-\frac{5}{8})}{\frac{1}{4}}, & \hbox{ $\frac{5N}{8}\le i \le \frac{7N}{8},$ }\\
\displaystyle1-\frac{8}{\theta_{1}}\phi_{4}(\xi_i), & \hbox{ $\frac{7N}{8}\le i \le N, $ }
\end{array}
\right.\]
with $\displaystyle \xi_{i}=\frac{i}{N}$. The functions $\phi_{1}$ and $ \phi_{3}$ are monotonically increasing on $[0, \frac{1}{8}]$ and $[\frac{1}{2}, \frac{5}{8}]$ respectively. The functions $\phi_{2}$ and $\phi_{4}$ are monotonically decreasing on $[\frac{3}{8}, \frac{1}{2}]$ and $[\frac{7}{8}, 1]$ respectively. 
\begin{Assumption}
	\label{assump1}
	Here we assume for $\sqrt{\alpha}\mu \le \sqrt{\rho\epsilon}$, $\sqrt{\epsilon}<N^{-1}$ to resolve the layers properly.
\end{Assumption}
\begin{lemma}
	\label{assump}
	Here, the mesh-generating functions $\phi_{1}, \phi_{2}, \phi_{3}$ and $\phi_{4}$ are piecewise differentiable and satisfy the following conditions:
	$$\max\limits_{\xi\in [0,\frac{1}{8}]}|\phi_{1}^{'}(\xi)|\le CN,~~~~\max\limits_{\xi\in [\frac{3}{8},\frac{1}{2}]}|\phi_{2}^{'}(\xi)|\le CN,$$
	$$\max\limits_{\xi\in [\frac{1}{2},\frac{5}{8}]}|\phi_{3}^{'}(\xi)|\le CN,~~~~\max\limits_{\xi\in [\frac{7}{8},1]}|\phi_{4}^{'}(\xi)|\le CN$$ 
	and
	$$\int_{0}^{\frac{1}{8}}\{\phi_{1}^{'}(\xi)\}^{2} d\xi\le CN,~~~~~\int_{\frac{3}{8}}^{\frac{1}{2}}\{\phi_{2}^{'}(\xi)\}^{2}d\xi\le CN,$$
	$$\int_{\frac{1}{2}}^{\frac{5}{8}}\{\phi_{3}^{'}(\xi)\}^{2}d\xi\le CN,~~~~~ \int_{\frac{7}{8}}^{1}\{\phi_{4}^{'}(\xi)\}^{2}d\xi\le CN.$$ 
\end{lemma}
\begin{proof}
  The mesh-generating functions $\phi_{1}(\xi)=-\log\bigg[1-8\xi\bigg(\frac{1}{\sqrt{N}}-1\bigg)\bigg],~~\xi\in [0,\frac{1}{8}].$\\
  Therefore, 
  $$\lvert \phi_{1}'(\xi)\rvert\le\frac{8\sqrt{N}}{\sqrt{N}+(1-\sqrt{N})}\le8\sqrt{N}\le CN.$$
  Similarly, we can prove the bounds for $\phi_{2}', \phi_{3}'$ and $\phi_{4}'$ in the intervals $[\frac{3}{8},\frac{1}{2}], [\frac{1}{2},\frac{7}{8}]$ and $[\frac{7}{8},1]$ respectively.
Also $$\int_{0}^{\frac{1}{8}}\{\phi_{1}^{'}(\xi)\}^{2} d\xi\le \int_{0}^{\frac{1}{8}}(8\sqrt{N})^{2} d\xi\le CN.$$
For other bounds, we can follow a similar procedure.
\end{proof}
Using the Lemma (\ref{assump}) and Assumption (\ref{assump1}) we see that for $1\le i \le \frac{N}{8}$, 
$$h_{i}=x_{i}-x_{i-1}=\frac{8}{\theta_{1}}\{\phi_{1}(\xi_{i})-\phi_{1}(\xi_{i-1})\}\le\frac{8}{\theta_{1}}(\xi_{i}-\xi_{i-1})\max\limits_{\xi\in [0,\frac{1}{8}]}|\phi_{1}^{'}(\xi)|\le \frac{C}{\theta_{1}}\le CN^{-1}.$$
Similarly, in the interval $\frac{3N}{8}\le i \le \frac{N}{2},\frac{N}{2}\le i \le \frac{5N}{8}$ and $\frac{7N}{8}\le i \le N$ we can bounds $h_{i}$ by using different $\phi_{i}'s, i=2,3,4$ to obtain that
\[h_{i}\le CN^{-1},~~\forall~x_{i}\in [d-\tau_{2},d]\cup[d,d+\tau_{3}]\cup[1-\tau_{4},1].\]

The fully discretized scheme is given by : 
Find $U^{j+1}(x_{i})=U(x_{i},t_{j+1})$ such that
\begin{equation}
\label{Full-dis}
\left\{
\begin{array}{ll}
\displaystyle \mathcal{L}^{N}U^{j+1}(x_{i})=\tilde{g}(x_{i},t_{j+1}), & \hbox{ $x_{i}\in \Gamma^{N-}\cup\Gamma^{N+},0\le j \le M-1,$ }\\
U^{j+1}(0)=u(0,t_{j+1}),U^{j+1}(1)=u(1,t_{j+1}), & \hbox{ $0\le j\le M-1$,}\\
U^{0}(x_{i})=u(x_{i},0), & \hbox{ $i=0,\ldots,N,$ }
\end{array}
\right.
\end{equation}
where the operator $\mathcal{L}^{N}$ is defined as\\
$$\mathcal{L}^{N}\equiv (\epsilon \delta_{x}^{2}+\mu a^{j+\frac{1}{2}}D_{x}^{*}-c^{j+\frac{1}{2}}I)$$
and $\tilde{g}(x_{i},t_{j+1})=2f^{j+\frac{1}{2}}(x_{i})-\epsilon \delta_{x}^{2}U^{j}(x_{i})-\mu a^{j+\frac{1}{2}}(x_{i})D_{x}^{*}U^{j}(x_{i})+d^{j+\frac{1}{2}}(x)U^{j}(x_{i}).$\\
At the point of discontinuity, we have used a three-point formula 
$$D_{x}^{+}U^{j+1}(x_{\frac{N}{2}})=D_{x}^{-}U^{j+1}(x_{\frac{N}{2}}),~\forall~0\le j\le M-1,$$
where $$D^{+}U^{j+1}(x_{i})=\frac{U^{j+1}(x_{i+1})-U^{j+1}(x_{i})}{x_{i+1}-x_{i}},~~~~~~ D^{-}U^{j+1}(x_{i})=\frac{U^{j+1}(x_{i})-U^{j+1}(x_{i-1})}{x_{i}-x_{i-1}}$$
$$D^{*}U^{j+1}(x_{i})=\left\{
\begin{array}{ll}
\displaystyle D^{-}U^{j+1}(x_{i}), & \hbox{ $i<\frac{N}{2}$ }\\
D^{+}U^{j+1}(x_{i}), & \hbox{ $i>\frac{N}{2}$,}
\end{array}
\right. ~~~\delta^{2}U^{j+1}(x_{i})=\frac{2(D^{+}U^{j+1}(x_{i})-D^{-}U^{j+1}(x_{i}))}{x_{i+1}-x_{i-1}}.$$
The matrix associated with the above discrete scheme (\ref{Full-dis})  is monotone and irreducibly diagonally dominant. It is an $M$-matrix and hence invertible.


\section{Convergence and Stability Analysis}
\label{4}
\begin{theorem}
	Suppose that a mesh function $Y(x_{i},t_{j})$ satisfies $Y(x_{0},t_{j})\ge 0, Y(x_{N},t_{j})\ge0$ and $D_{x}^{+}Y(x_{\frac{N}{2}},t_{j})-D_{x}^{-}Y(x_{\frac{N}{2}},t_{j})\le0 ~\forall~j=0,\ldots,M$. If $\mathcal{L}^{N}Y(x_{i},t_{j})\le0$ for all $(x_{i},t_{j})\in \Omega^{N-}\cup\Omega^{N+}$ then $Y(x_{i},t_{j})\ge0,~\forall  (x_{i},t_{j})\in\bar{\Omega}^{N,M}$.
\end{theorem}
\begin{proof}
See \cite{DP} for proof.
\end{proof}
To find the error estimates for the scheme \eqref{Full-dis} defined  above, we first decompose the discrete solution $U^{j+1}(x_{i})$ into the regular and singular components. \\
Let
\[U^{j+1}(x_{i})=V^{j+1}(x_{i})+W_{l}^{j+1}(x_{i})+W_{r}^{j+1}(x_{i}).\]
The regular components is
\begin{equation}
\begin{aligned}
\label{reg}
V^{j+1}(x_{i})=\left\{
\begin{array}{ll}
\displaystyle V^{-(j+1)}(x_{i}), & \hbox{ $(x_{i},t_{j})\in \Omega^{N-},$ }\\
V^{+(j+1)}(x_{i}),& \hbox{ $(x_{i},t_{j})\in \Omega^{N+},$ }
\end{array}
\right.
\end{aligned}
\end{equation}
where $V^{-(j+1)}(x_{i})$ and $V^{+(j+1)}(x_{i})$ approximate $v^{-}(x_{i},t_{j+1})$ and $v^{+}(x_{i},t_{j+1})$ respectively. They satisfy the following equations:
\begin{equation}
\begin{aligned}
\label{LV-}
&\mathcal{L}^{N}V^{-(j+1)}(x_{i})=\tilde{g}(x_{i},t_{j+1}),~\forall~ (x_{i},t_{j+1})\in\Omega^{N-},\\
&V^{-(j+1)}(x_{0})=v^{-}(0,t_{j+1}),~ V^{-(j+1)}(x_{\frac{N}{2}})=v^{-}(d-,t_{j+1}),~\forall ~j=0,1,\ldots,M-1.
\end{aligned}
\end{equation}
\begin{equation}
\begin{aligned}
\label{LV+}
&\mathcal{L}^{N}V^{+(j+1)}(x_{i})=g(x_{i},t_{j+1}),~\forall~ (x_{i},t_{j+1})\in\Omega^{N+},\\
&V^{+(j+1)}(x_{\frac{N}{2}})=v^{+}(d+,t_{j+1}), V^{+(j+1)}(x_{N})=v^{-}(1,t_{j+1}),~\forall ~j=0,1,\ldots,M-1.
\end{aligned}
\end{equation}
The singular components $W_{l}^{j+1}(x_{i})$ and $W_{r}^{j+1}(x_{i})$ are also decomposed as:
$$W_{l}^{j+1}(x_{i})=\left\{
\begin{array}{ll}
\displaystyle W_{l}^{-(j+1)}(x_{i}), & \hbox{ $(x_{i},t_{j})\in \Omega^{N-},$ }\\
W_{l}^{+(j+1)}(x_{i}),& \hbox{ $(x_{i},t_{j})\in \Omega^{N+},$ }
\end{array}
\right.
W_{r}^{j+1}(x_{i})=\left\{
\begin{array}{ll}
\displaystyle W_{r}^{-(j+1)}(x_{i}), & \hbox{$(x_{i},t_{j})\in \Omega^{N-},$ }\\
W_{r}^{+(j+1)}(x_{i}),& \hbox{$(x_{i},t_{j})\in \Omega^{N+},$ }
\end{array}
\right.$$
where $W_{l}^{-(j+1)}(x_{i}), W_{l}^{+(j+1)}(x_{i})$ approximates the left layer components $w_{l}^{-}(x_{i},t_{j+1})$ and $w_{l}^{+}(x_{i},t_{j+1})$ and $W_{r}^{-(j+1)}(x_{i}), W_{r}^{+(j+1)}(x_{i})$ approximates the right layer components $w_{r}^{-}(x_{i},t_{j+1})$ and $w_{r}^{+}(x_{i},t_{j+1})$ respectively.   
These components satisfies  the following equations
\begin{equation}
\begin{aligned}
\label{LW-}
&\mathcal{L}^{N}W_{l}^{-(j+1)}(x_{i})=0,~\forall~ (x_{i},t_{j+1})\in\Omega^{N-},\\
&W_{l}^{-(j+1)}(x_{0})=w_{l}^{-}(0,t_{j+1}), W_{l}^{-(j+1)}(x_{\frac{N}{2}})=w_{l}^{-}(d,t_{j+1}),~\forall ~j=0,1,\ldots,M-1.
\end{aligned}
\end{equation}
\begin{equation}
\begin{aligned}
\label{LW+}
&\mathcal{L}^{N}W_{l}^{+(j+1)}(x_{i})=0,~\forall~ (x_{i},t_{j+1})\in\Omega^{N+},\\
&W_{l}^{+(j+1)}(x_{\frac{N}{2}})=w_{l}^{+}(d,t_{j+1}),~ W_{l}^{+(j+1)}(x_{N})=0,~\forall ~j=0,1,\ldots,M-1.
\end{aligned}
\end{equation}
\begin{equation}
\begin{aligned}
\label{RW-}
&\mathcal{L}^{N}W_{r}^{-(j+1)}(x_{i})=0,~\forall~ (x_{i},t_{j+1})\in\Omega^{N-},\\
&W_{r}^{-(j+1)}(x_{0})=0,~ W_{r}^{-(j+1)}(x_{\frac{N}{2}})=w_{r}^{-}(d,t_{j+1}), ~\forall ~j=0,1,\ldots,M-1.
\end{aligned}
\end{equation}
\begin{equation}
\begin{aligned}
\label{RW+}
&\mathcal{L}^{N}W_{r}^{+(j+1)}(x_{i})=0,~\forall~ (x_{i},t_{j+1})\in\Omega^{N+},\\
&W_{r}^{+(j+1)}(x_{\frac{N}{2}})=0,~ W_{r}^{+(j+1)}(x_{N})=w_{r}(1,t_{j+1}),~\forall ~j=0,1,\ldots,M-1.
\end{aligned}
\end{equation}
Hence, the discrete solution $U^{j+1}(x_{i})$ is defined as
$$U^{j+1}(x_{i})=\left\{
\begin{array}{ll}
\vspace{0.3cm}
\displaystyle (V^{-(j+1)}+W_{l}^{-(j+1)}+W_{r}^{-(j+1)})(x_{i}),\hspace{5cm} (x_{i},t_{j+1})\in\Omega^{N-},\\

(V^{-(j+1)}+W_{l}^{-(j+1)}+W_{r}^{-(j+1)})(x_{\frac{N}{2}})\\
\vspace{0.3cm}
=(V^{+(j+1)}+W_{l}^{+(j+1)}+W_{r}^{+(j+1)})(x_{\frac{N}{2}}),\hspace{4.4cm}(x_{i},t_{j+1})=(d,t_{j+1}),\\

(V^{+(j+1)}+W_{l}^{+(j+1)}+W_{r}^{+(j+1)})(x_{i}),\hspace{5cm}(x_{i},t_{j+1})\in\Omega^{N+}.
\end{array}
\right.$$
\begin{theorem}
	Let $\sqrt{\alpha}\mu\le\sqrt{\gamma\epsilon}$, the singular component $W_{l}^{-(j+1)}(x_{i}), W_{l}^{+(j+1)}(x_{i})$, $W_{r}^{-(j+1)}(x_{i})$ and $ W_{r}^{+(j+1)}(x_{i})$ satisfy the following bounds
	$$|W_{l}^{-(j+1)}(x_{i})|\le C\gamma_{l,i}^{-(j+1)},\quad\gamma_{l,i}^{-(j+1)}=\prod_{n=1}^{i}(1+\theta_{1}h_{n})^{-1},~~\gamma_{l,0}^{-(j+1)}=C_{1},~i=0,1,\ldots,\frac{N}{2},$$
	$$|W_{l}^{+(j+1)}(x_{i})|\le C\gamma_{l,i}^{+(j+1)},\quad\gamma_{l,i}^{+(j+1)}=\prod_{n=\frac{N}{2}+1}^{i}(1+\theta_{2}h_{n})^{-1},~~\gamma_{l,\frac{N}{2}}^{+(j+1)}=C_{1},~i=\frac{N}{2}+1,\ldots,N,$$
	$$|W_{r}^{-(j+1)}(x_{i})|\le C\gamma_{r,i}^{-(j+1)},\quad\gamma_{r,i}^{-(j+1)}=\prod_{n=i+1}^{\frac{N}{2}}(1+\theta_{2}h_{n})^{-1},~~\gamma_{l,\frac{N}{2}}^{-(j+1)}=C_{1},~i=0,1,\ldots,\frac{N}{2},$$
	$$|W_{r}^{+(j+1)}(x_{i})|\le C\gamma_{r,i}^{+(j+1)},\quad\gamma_{r,i}^{+(j+1)}=\prod_{n=i+1}^{N}(1+\theta_{1}h_{n})^{-1},~~\gamma_{l,N}^{+(j+1)}=C_{1},~i=\frac{N}{2}+1,\ldots,N.$$
	\begin{equation}
\label{theta}
	\theta_{1}=\frac{\sqrt{\rho \alpha}}{2\sqrt{\epsilon}},~~\theta_{2}=\frac{\sqrt{\rho \alpha}}{2\sqrt{\epsilon}}.
\end{equation}
\end{theorem}
\begin{proof}
	Let us define the barrier function for the left layer term $W_{l}^{-(j+1)}(x_{i})$  in $\Omega^{N-}$ as
	$$\psi_{l}^{-(j+1)}(x_{i})=\gamma_{l,i}^{-(j+1)} \pm W_{l}^{-(j+1)}(x_{i})$$
	where 
	$$\gamma_{l,i}^{-(j+1)}=\left\{
	\begin{array}{ll}
	\displaystyle \prod_{n=1}^{i}(1+\theta_{1}h_{n})^{-1},~~1\le i\le\frac{N}{2},\\
C_{1},\hspace{3cm}i=0,
	\end{array}
	\right.$$
	$ $\\
	and $\theta_{1}$ is defined in \eqref{theta} of Section \ref{4}.
	For large C, $\psi_{l}^{-(j+1)}(x_{0})\ge0$ and $\psi_{l}^{-(j+1)}(x_{\frac{N}{2}})\ge0,~~\forall~j=0,\ldots,M-1.$ 
	
	Consider,
	\begin{align*}
	\mathcal{L}^{N}\psi_{l}^{-(j+1)}(x_{i})&=\{\epsilon \delta^{2}_{x}\gamma_{l,i}^{-(j+1)}+\mu a^{j+\frac{1}{2}}(x_{i})D^{-}_{x}\gamma_{l,i}^{-(j+1)}-c^{j+\frac{1}{2}}(x_{i})\gamma_{l,i}^{-(j+1)}\}\\&=\big\{\epsilon \frac{\theta_{1}^{2}}{\hbar_{i}}\gamma_{l,i+1}^{-(j+1)}h_{i+1}+\mu a^{j+\frac{1}{2}}(x_{i})(-\theta_{1}\gamma_{l,i}^{-(j+1)})-c^{j+\frac{1}{2}}(x_{i})\gamma_{l,i}^{-(j+1)}\big\}~~\bigg(\hbar_{i}=\frac{h_{i}+h_{i+1}}{2}\bigg)\\&=\gamma_{l,i+1}^{-(j+1)}\bigg[\epsilon\theta_{1}^{2}\bigg(\frac{h_{i+1}}{\hbar_{i}}-2\bigg)+2\epsilon\theta_{1}^{2}-\mu a^{j+\frac{1}{2}}(x_{i})\theta_{1}-c^{j+\frac{1}{2}}(x_{i})\\&-\mu a^{j+\frac{1}{2}}(x_{i})\theta_{1}^{2}h_{i+1}-c^{j+\frac{1}{2}}(x_{i})\theta_{1}h_{i+1}\bigg]\\&\le \gamma_{l,i+1}^{-(j+1)}\big(2\epsilon\theta_{1}^{2}-\mu a^{j+\frac{1}{2}}(x_{i})\theta_{1}-c^{j+\frac{1}{2}}(x_{i})\big)\\&\le \gamma_{l,i+1}^{-(j+1)}\bigg(2\epsilon \frac{\rho\alpha}{4\epsilon}-\mu a^{j+\frac{1}{2}}(x_{i})\frac{\sqrt{\rho\alpha}}{2\sqrt{\epsilon}}-c^{j+\frac{1}{2}}(x_{i})\bigg)\\&\le \gamma_{l,i+1}^{-(j+1)}(\rho|a^{j+\frac{1}{2}}(x_{i})|-c^{j+\frac{1}{2}}(x_{i}))\le 0
	\end{align*}
	Therefore, by the discrete minimum principle defined in \cite{ER1} for the continuous case, we prove that
	$\psi_{l}^{-(j+1)}(x_{i})\ge 0$, $\forall~x\in \Omega^{N-}.$\\
Similarly, we define the barrier function for right layer term $W_{r}^{-(j+1)}(x_{i})$ in $\Omega^{N-}$ as
	$$\psi_{r}^{-(j+1)}(x_{i})=C\gamma_{r,i}^{-(j+1)} \pm W_{r}^{-(j+1)}(x_{i})$$
	where
	$$\gamma_{r,i}^{-(j+1)}=\left\{
	\begin{array}{ll}
	\displaystyle \prod_{n=i+1}^{\frac{N}{2}}(1+\theta_{2}h_{n})^{-1},~~0\le i\le\frac{N}{2}-1,\\
C_{1},\hspace{3cm}i=\frac{N}{2},
	\end{array}
	\right.$$ \\
	where $\theta_{2}$ is defined in \eqref{theta} of Section \ref{4}.
	For large C, $\psi_{r}^{-(j+1)}(x_{0})\ge0$ and $\psi_{r}^{-(j+1)}(x_\frac{N}{2})\ge0,~~\forall~j=0,\ldots,M-1.$ 
	
	Consider
	\begin{align*}
	\mathcal{L}^{N}\psi_{r}^{-(j+1)}(x_{i})&=C\big\{\epsilon \delta^{2}_{x}\gamma_{r,i}^{-(j+1)}+\mu a^{j+\frac{1}{2}}(x_{i})D^{-}_{x}\gamma_{r,i}^{-(j+1)}-c^{j+\frac{1}{2}}(x_{i})\gamma_{r,i}^{-(j+1)}\big\}\\&=C\bigg\{\epsilon \frac{\theta_{2}^{2}h_{i}}{\hbar_{i}}\frac{\gamma_{r,i}^{-(j+1)}}{1+\theta_{2}h_{i}}+\mu a^{j+\frac{1}{2}}(x_{i})(\theta_{2}\gamma_{r,i-1}^{-(j+1)})-c^{j+\frac{1}{2}}(x_{i})\gamma_{r,i}^{-(j+1)}\bigg\}\\&=\frac{\gamma_{r,i}^{-(j+1)}}{1+\theta_{2}h_{i}}\bigg[\epsilon\frac{\theta_{2}^{2}h_{i}}{\hbar_{i}}+\mu a^{j+\frac{1}{2}}(x_{i})\theta_{2}-c^{j+\frac{1}{2}}(x_{i})(1+\theta_{2}h_{i})\bigg]\\&=\frac{\gamma_{r,i}^{-(j+1)}}{1+\theta_{2}h_{i}}\bigg[\epsilon \theta_{2}^{2}\bigg(\frac{h_{i}}{\hbar_{i}}-2\bigg)+2\epsilon \theta_{2}^{2}+\mu a^{j+\frac{1}{2}}(x_{i})\theta_{2}-c^{j+\frac{1}{2}}(x_{i})(1+\theta_{2}h_{i})\bigg]\\&\le\frac{\gamma_{r,i}^{-(j+1)}}{1+\theta_{2}h_{i}}\big(2\epsilon\theta_{2}^{2}+\mu a^{j+\frac{1}{2}}(x_{i})\theta_{2}-c^{j+\frac{1}{2}}(x_{i})\big)\\&\le\frac{\gamma_{r,i}^{-(j+1)}}{1+\theta_{2}h_{i}}\big(2\epsilon\frac{\rho\alpha}{4\epsilon}+\mu a^{j+\frac{1}{2}}(x_{i})\frac{\sqrt{\rho\alpha}}{2\sqrt{\epsilon}}-c^{j+\frac{1}{2}}(x_{i})\big)\\&=\frac{\gamma_{r,i}^{-(j+1)}}{1+\theta_{2}h_{i}}(-c^{j+\frac{1}{2}}(x_{i}))\le 0
	\end{align*}
	
Therefore, by the discrete minimum principle defined in \cite{ER1} for the continuous case, we prove that
$\psi_{r}^{-(j+1)}(x_{i})\ge 0$,~$\forall~x\in \Omega^{N-}.$
Similarly, by defining the corresponding barrier functions $\psi_{l}^{+(j+1)}(x_{i})$ and $\psi_{r}^{+(j+1)}(x_{i})$ for $W_{l}^{+(j+1)}(x_{i})$ and $W_{r}^{+(j+1)}(x_{i})$ we get the remaining two inequalities for the left and right layer term in $\Omega^{N+}.$
\end{proof}
\begin{theorem}
	The discrete regular component $V^{j+1}(x_{i})$  defined in \eqref{reg} and $ v(x,t)$ is solution of the problem \eqref{v}. So, the error in the regular component satisfies the following estimate for $\sqrt{\alpha}\mu\le \sqrt{\rho \epsilon}$ :
	$$\|V-v\|_{\Omega^{N-}\cup\Omega^{N+}}\le C(N^{-1}+(\Delta t)^{2}).$$
\end{theorem}
\begin{proof}
		Using the truncation error in domain $\Omega^{N-}$, we get
	\begin{align*}
	|\mathcal{L}^{N}(V^{-(j+1)}-v^{-(j+1)})(x_{i})|&=|\mathcal{L}^{N}V^{-(j+1)}(x_{i})-\mathcal{L}^{N}v^{-(j+1)}(x_{i})|,~~(x_{i},t_{j+1})\in\Omega^{N-}\\
	&\le \epsilon \bigg|\bigg(\delta^{2}_{x}-\frac{d^{2}}{dx^{2}}\bigg) v^{-(j+1)}(x_{i})\bigg|+\mu |a^{j+\frac{1}{2}}(x_{i})|\bigg|\bigg(D^{-}_{x}-\frac{d}{dx}\bigg)v^{-(j+1)}(x_{i})\bigg|\\
	&+\bigg|D_{t}^{-}u^{-(j+1)}(x_{i})-\frac{\delta}{\delta t}u^{-(j+\frac{1}{2})}(x_{i})\bigg|\\
	&\le C_{1} \max_{0\le i\le\frac{N}{2}} h_{i}(\epsilon\| v^{-}_{xxx}\|_{\Omega^{-}}+\mu\|v^{-}_{xx}\|_{\Omega^{-}})+C_{2}(\Delta t)^{2}\\
	&\le C(N^{-1}+(\Delta t)^{2}).
	\end{align*}
Define the barrier function
	$$\psi^{j+1}(x_{i})=C(N^{-1}+(\Delta t)^{2})\pm(V^{-(j+1)}-v^{-(j+1)})(x_{i}),~~~(x_{i},t_{j+1})\in\Omega^{N-}.$$
	It is clear that $\psi^{j+1}(x_{0})\ge0$ and $\psi^{j+1}(x_{\frac{N}{2}})\ge0$. For large C, we obtain $~\mathcal{L}^N\psi^{j+1}(x_{i})\le0$. 
	Applying discrete minimum principle \cite{ER1}, we get $\psi^{j+1}(x_{i})\ge0.$
	\begin{equation}
	\label{VV-}
	\|V^{-}-v^{-}\|_{\Omega^{N-}}\le C(N^{-1}+(\Delta t)^{2}),~~~(x_{i},t_{j+1})\in\Omega^{N-}.
	\end{equation} 
	The error estimate in the domain $\Omega^{N+}$ is derived similarly.
	\begin{equation}
	\label{VV+}
	\|V^{+}-v^{+}\|_{\Omega^{N+}}\le C(N^{-1}+(\Delta t)^{2}),~~~(x_{i},t_{j+1})\in\Omega^{N+}.
	\end{equation} 
Combining the above two equations \eqref{VV-} and \eqref{VV+}, we obtain the desired result.	
\end{proof}
\begin{lemma}
	\label{L1}
	Let $\sqrt{\alpha}\mu\le \sqrt{\rho \epsilon}$ and $W_{l}^{-(j+1)}(x_{i}), w_{l}^{-}(x,t)$ are solution of the problem \eqref{LW-} and \eqref{lw} respectively. The left singular component of the truncation error satisfies the following estimate in $\Omega^{N-}$
	$$\|W_{l}^{-}-w_{l}^{-}\|_{\Omega^{N-}}\le C(N^{-1}+(\Delta t)^{2}).$$
\end{lemma}
\begin{proof}
	We first calculate the truncation error in the outer region $[\tau_{1},d)\times(0, T]$:
	In $[\tau_{1},d)\times(0, T]$,
	the left layer component satisfies the following bound given in theorem \eqref{wlb}:
	\begin{equation}
	\label{wl}
	|w_{l}^{-(j+1)}(x_{i})|\le C\exp^{-\theta_{1}x_{i}}\le C\exp^{- \theta_{1} \tau_{1}}\le CN^{-4},~~0\le j\le M-1.
	\end{equation}
	Also  $W_{l}^{-(j+1)}(x_{i})$ is decreasing function in $[\tau_{1},d)\times(0, T]$, so
	\begin{align*}
	|W^{-(j+1)}_{l}(x_{i})|&\le|W_{l}^{-(j+1)}(x_{\frac{N}{8}})|\\&= \prod_{n=1}^{\frac{N}{8}}(1+\theta_{1}h_{n})^{-1}\\ \log|W^{-(j+1)}_{l}(x_{i})|&\le \log\bigg(\prod_{n=1}^{\frac{N}{8}}(1+\theta_{1}h_{n})^{-1}\bigg)\\&=-\sum_{n=1}^{\frac{N}{8}}\log (1+\theta_{1}h_{n})\\ \text{As}~ \log(1+t^2)&\ge t-\frac{t^2}{2}~ \text{for}~ t\ge0.\\
	\text{So,}~ \sum_{n=1}^{\frac{N}{8}}\log (1+\theta_{1}h_{n})&\ge \sum_{n=1}^{\frac{N}{8}}\theta_{1}h_{n}-\sum_{n=1}^{\frac{N}{8}}\bigg(\frac{\theta_{1} h_{n}}{2}\bigg)^2\\&\ge 4\ln N-\sum_{n=1}^{\frac{N}{8}}\bigg(\frac{\theta_{1} h_{n}}{2}\bigg)^2.
	\end{align*}
	To calculate $\sum_{n=1}^{\frac{N}{8}}\bigg(\frac{\theta_{1} h_{n}}{2}\bigg)^2$, let $1\le n \le \frac{N}{8},$\\
	\begin{align*}
	h_{n}=x_{n}-x_{n-1}&=\frac{8}{\theta_{1}}(\phi_{1}(\xi_{n})-\phi_{1}(\xi_{n-1}))\\&=\frac{8}{\theta_{1}}\int_{\xi_{n-1}}^{\xi_{n}}\phi_{1}'(\xi)d\xi\\\frac{\theta_{1}h_{n}}{8}&=\int_{\xi_{n-1}}^{\xi_{n}}\phi_{1}'(\xi)d\xi\\\implies
	\bigg(\frac{\theta_{1}h_{n}}{8}\bigg)^2 &\le(\xi_{n}-\xi_{n-1})\int_{\xi_{n-1}}^{\xi_{n}}\phi_{1}'(\xi)^{2}d\xi~~~\text{(Using Holder's inequality)}\\
	\sum_{n=1}^{\frac{N}{8}}\bigg(\frac{\theta_{1}h_{n}}{8}\bigg)^2&\le N^{-1}\int_{0}^{\frac{1}{8}}\phi_{1}'(\xi)^{2}d\xi\\&\le C ~~~(\text{from Lemma }~\ref{assump}).
	\end{align*}
Hence,\\
\begin{align}
\label{W-}
|W_{l}^{-(j+1)}(x_{i})|\le CN^{-4},~(x_{i},t_{j+1})\in [\tau_{1},d)\times(0, T].
\end{align}
By combining equation (\ref{wl}) and (\ref{W-}), we get
\begin{equation}
\label{ww}
|(W_{l}^{-(j+1)}-w_{l}^{-(j+1)})(x_{i})|\le |W_{l}^{-(j+1)}(x_{i})|+|w_{l}^{-(j+1)}(x_{i})|\le CN^{-4},~~(x_{i},t_{j+1})\in [\tau_{1},d)\times(0, T].
\end{equation}

We use truncation error analysis to find error in the inner region $(0,\tau_{1})\times(0,T]$. Using the derivative bounds for the left layer component $w_{l}^{-}$ from the theorem \eqref{wwlbd}, we obtain
\begin{align*}
|\mathcal{L}^{N}(W^{-(j+1)}_{l}-w^{-(j+1)}_{l})(x_{i})|
&\le \epsilon \bigg|\bigg(\delta^{2}_{x}-\frac{d^{2}}{dx^{2}}\bigg) w^{-(j+1)}_{l}(x_{i})\bigg|+\mu |a^{j+\frac{1}{2}}(x_{i})|\bigg|\bigg(D^{-}_{x}-\frac{d}{dx}\bigg)w^{-(j+1)}_{l}(x_{i})\bigg|\\&+\bigg|D_{t}^{-}w_{l}^{-(j+1)}(x_{i})-\frac{\delta}{\delta t}w_{l}^{-(j+\frac{1}{2})}(x_{i})\bigg|\\
&\le C_{1} \max_{0\le i\le\frac{N}{2}} h_{i}(\epsilon\|w^{-}_{xxx}\|_{\Omega^{-}}+\mu\|w^{-}_{xx}\|_{\Omega^{-}})+C_{2}(\Delta t)^{2})\\
& \le \frac{C_{1}}{\theta_{1}}\bigg(\frac{1}{\sqrt{\epsilon}}\bigg)+C_{2}(\Delta t)^{2}~~(\text{from theorem } \eqref{wwlbd})\\
&\le C\bigg(\frac{N^{-1}}{\sqrt{\epsilon}}+(\Delta t)^2\bigg).
\end{align*}
	Choosing the barrier function for the layer component as
$$\psi^{j+1}(x_{i})=C_{1}\bigg(\frac{x_{i}N^{-1}}{\sqrt{\epsilon} \ln N}+(\Delta t)^{2}\bigg)\pm(W_{l}^{-(j+1)}-w_{l}^{-(j+1)})(x_{i}),~~(x_{i},t_{j+1})\in(0,\tau_{1})\times(0,T].$$
For sufficiently large $C_{1}$, we have $\mathcal{L}^{N}\psi^{j+1}(x_{i})\le0,~(x_{i},t_{j+1})\in (0,\tau_{1})\times(0,T]$. Also $\psi^{j+1}(x_{0})\ge0$ and $\psi^{j+1}(x_{\frac{N}{8}})\ge0$. Hence by discrete minimum principle \cite{ER1}, we can obtain the following bounds:
\begin{align*}
|(W_{l}^{-(j+1)}-w_{l}^{-(j+1)})(x_{i})|&\le C_{1}\bigg(\frac{x_{i}N^{-1}}{\sqrt{\epsilon} \ln N}+(\Delta t)^{2}\bigg)\\&= \bigg(\frac{\tau_{1}N^{-1}}{\sqrt{\epsilon} \ln N}+(\Delta t)^{2}\bigg)
\end{align*}
\begin{equation}
\label{WW}
|(W_{l}^{-(j+1)}-w_{l}^{-(j+1)})(x_{i})|\le C(N^{-1}+(\Delta t)^{2}).
\end{equation}
Hence, by combining equation \eqref{ww} and \eqref{WW}, we have bounds for the left layer component
\begin{equation}
\label{ww-}
\|W_{l}^{-}-w_{l}^{-}\|_{\Omega^{N-}}\le C(N^{-1}+(\Delta t)^{2}).
\end{equation}
\end{proof}
\begin{lemma}
	\label{L2}
	Let $\sqrt{\alpha}\mu\le \sqrt{\rho \epsilon}$ and $W_{l}^{+(j+1)}(x_{i}), w_{l}^{+}(x,t)$ be the solution of the problem \eqref{LW+} and \eqref{lw} respectively. The left singular component of the truncation error satisfies the following estimate in $\Omega^{N+}$
	$$\|W_{l}^{+}-w_{l}^{+}\|_{\Omega^{N+}}\le CN^{-1}+(\Delta t)^{2})$$
\end{lemma}
\begin{proof}
	We first calculate the truncation error in outer region $[d+\tau_{3},1)\times(0, T]$:
	
	In $[d+\tau_{3},1)\times(0, T]$ i.e. $\frac{5N}{8}\le i < N$, the left layer component satisfies the following bound given in theorem \eqref{wlb}:
	\begin{equation}
	\label{wl+}
	|w_{l}^{+(j+1)}(x_{i})|\le Ce^{-\theta_{2}(x_{i}-d)}\le Ce^{-\theta_{2}\tau_{3}}\le CN^{-4},~~(x_{i},t_{j+1})\in[d+\tau_{3},1)\times(0, T].
	\end{equation}
	Also $W^{+(j+1)}_{l}(x_{i})$ is a decreasing function in $[d+\tau_{3},1)\times(0, T]$, so
	\begin{align*}
	|W^{+(j+1)}_{l}(x_{i})|&\le|W_{l}^{+(j+1)}(x_{\frac{5N}{8}})|= \prod_{n=\frac{N}{2}+1}^{\frac{5N}{8}}(1+\theta_{2}h_{n})^{-1}
	\end{align*}
Applying the similar arguments as in Lemma \eqref{L1}, we get
	\begin{align}
	\label{W+} 
	|W_{l}^{+(j+1)}(x_{i})|\le CN^{-4},~~(x_{i},t_{j+1})\in[d+\tau_{3},1)\times(0, T].
	\end{align}
Hence, by combining equation (\ref{wl+}) and (\ref{W+}) we get
	\begin{equation}
	\label{www}
	|(W_{l}^{+(j+1)}-w_{l}^{+(j+1)})(x_{i})|\le |W_{l}^{+(j+1)}(x_{i})|+|w_{l}^{+(j+1)}(x_{i})|\le CN^{-4},~(x_{i},t_{j+1})\in[d+\tau_{3},1)\times(0, T].
	\end{equation}
	
	We use truncation error analysis to find errors in the inner region $(d,d+\tau_{3})\times(0,T])$. Using the derivative bounds for the left layer component $w_{l}^{+}$ from the theorem \eqref{wwlbd}, we obtain
	\begin{align*}
	|\mathcal{L}^{N}(W^{+(j+1)}_{l}-w^{+(j+1)}_{l})(x_{i})|
	&\le \epsilon \bigg|\bigg(\delta^{2}_{x}-\frac{d^{2}}{dx^{2}}\bigg) w^{+(j+1)}_{l}(x_{i})\bigg|+\mu |a^{j+\frac{1}{2}}(x_{i})|\bigg|\bigg(D^{+}_{x}-\frac{d}{dx}\bigg)w^{+(j+1)}_{l}(x_{i})\bigg|\\&+\bigg|D_{t}^{-}w_{l}^{+(j+1)}(x_{i})-\frac{\delta}{\delta t}w_{l}^{+(j+\frac{1}{2})}(x_{i})\bigg|\\
	&\le C_{1} \max_{\frac{N}{2}\le i\le N} h_{i}(\epsilon\|w^{+}_{xxx}\|_{\Omega^{+}}+\mu\|w^{+}_{xx}\|_{\Omega^{+}})+C_{2}(\Delta t)^{2})\\
	& \le \frac{C_{1}}{\theta_{2}}\bigg(\frac{1}{\sqrt{\epsilon}}\bigg)+C_{2}(\Delta t)^{2}\le C\bigg(\frac{N^{-1}}{\sqrt{\epsilon}}+(\Delta t)^2\bigg).
	\end{align*}
	Choosing the barrier function for the layer component as
	$$\psi^{j+1}(x_{i})=C_{1}\bigg(\frac{(d+\tau_{3}-x_{i})N^{-1}}{\sqrt{\epsilon} \ln N}+(\Delta t)^{2}\bigg)\pm(W_{l}^{+(j+1)}-w_{l}^{+(j+1)})(x_{i}),~(x_{i},t_{j+1})\in(d,d+\tau_{3})\times(0,T]$$
	For sufficiently large $C_{1}$, we have $\mathcal{L}^{N}\psi^{j+1}(x_{i})\le0,~(x_{i},t_{j+1})\in(d,d+\tau_{3})\times(0,T]$. Also $\psi^{j+1}(x_{\frac{N}{2}})\ge0$ and $\psi^{j+1}(x_{\frac{5N}{8}})\ge0$. Hence, by the discrete minimum principle \cite{ER1}, we can obtain the following bounds:
	\begin{align*}
	|(W_{l}^{+(j+1)}-w_{l}^{+(j+1)})(x_{i})|&\le C_{1} \bigg(\frac{(d+\tau_{3}-x_{i})N^{-1}}{\sqrt{\epsilon} \ln N}+(\Delta t)^{2}\bigg)\\& \le C_{1}\bigg(\frac{\tau_{3}N^{-1}}{\sqrt{\epsilon} \ln N}+(\Delta t)^{2}\bigg)
	\end{align*}
	\begin{equation}
	\label{WWW}
	|(W_{l}^{+(j+1)}-w_{l}^{+(j+1)})(x_{i})|\le C(N^{-1}+(\Delta t)^{2}),~(x_{i},t_{j+1})\in(d,d+\tau_{3})\times(0,T].
	\end{equation}
	Hence, by combining the equation\eqref{www} and \eqref{WWW}, we have bounds for the left layer component
	\begin{equation}
	\label{w+}
	\|W_{l}^{+}-w_{l}^{+}\|_{\Omega^{N+}}\le C(N^{-1}+(\Delta t)^{2}).
	\end{equation}
\end{proof}
\begin{lemma}
	\label{L3}
	Let $\sqrt{\alpha}\mu\le \sqrt{\rho \epsilon}$ and $W_{r}^{-(j+1)}(x_{i}), w_{r}^{-}(x,t)$ are solution of the problem \eqref{RW-} and \eqref{rw} respectively. The right singular component of the truncation error satisfies the following estimate in ${\Omega^{N-}}$
	$$\|W_{r}^{-}-w_{r}^{-}\|_{\Omega^{N-}}\le C(N^{-1}+(\Delta t)^{2})$$
\end{lemma}
\begin{proof}
We first calculate the truncation error in the outer region $(0,d-\tau_{2}]\times(0, T]$: For $1\le i\le \frac{3N}{8}$, the left layer component satisfies the following bound given in theorem \eqref{wlb}:
	\begin{equation}
	\label{wR}
	|w_{r}^{-(j+1)}(x_{i})|\le Ce^{-\theta_{2}(d-x_{i})}\le Ce^{-\theta_{2}\tau_{2}}\le CN^{-4},~(x_{i},t_{j+1})\in (0,d-\tau_{2}]\times(0, T].
	\end{equation}
	Also $W^{-(j+1)}_{r}(x_{i})$ is a increasing function in $(0,d-\tau_{2}]\times(0,T]$,
	\begin{align*}
	|W^{-(j+1)}_{r}(x_{i})|&\le|W_{r}^{-(j+1)}(x_{\frac{3N}{8}})|= \prod_{n=\frac{3N}{8}+1}^{\frac{N}{2}}(1+\theta_{2}h_{n})^{-1}.
	\end{align*}
	Applying the similar arguments as in Lemma \eqref{L1}, we get
	\begin{align}
	\label{wr}
	|W_{r}^{-(j+1)}(x_{i})|\le CN^{-4},(x_{i},t_{j+1})\in (0,d-\tau_{2}]\times(0, T].
	\end{align}
	Hence, by combining the equation (\ref{wR}) and (\ref{wr}), we get
	\begin{equation}
	\label{wwww}
	|(W_{r}^{-(j+1)}-w_{r}^{-(j+1)})(x_{i})|\le |W_{r}^{-(j+1)}(x_{i})|+|w_{r}^{-(j+1)}(x_{i})|\le CN^{-4},~(x_{i},t_{j+1})\in (0,d-\tau_{2}]\times(0, T].
	\end{equation}
	
We use truncation error analysis to find an error in the inner region $(d-\tau_{2},d)\times(0, T]$. Using the derivative bounds for the left layer component $w_{r}^{-}$ from theorem \eqref{wwlbd}, we obtain
	\begin{align*}
	|\mathcal{L}^{N}(W^{-(j+1)}_{r}-w^{-(j+1)}_{r})(x_{i})|
	&\le \epsilon \bigg| \bigg(\delta^{2}_{x}-\frac{d^{2}}{dx^{2}}\bigg) w^{-(j+1)}_{r}(x_{i})\bigg|+\mu |a^{j+\frac{1}{2}}(x_{i})|\bigg|\bigg(D^{-}_{x}-\frac{d}{dx}\bigg)w^{-(j+1)}_{r}(x_{i})\bigg|\\&+\bigg|D_{t}^{-}w_{r}^{-(j+1)}(x_{i})-\frac{\delta}{\delta t}w_{r}^{-(j+\frac{1}{2})}(x_{i})\bigg|\\
	&\le C_{1} \max_{0\le i\le \frac{N}{2}} h_{i}(\epsilon\|w^{-}_{xxx}\|_{\Omega^{-}}+\mu\|w^{-}_{xx}\|_{\Omega^{-}})+C_{2}(\Delta t)^{2})\\
	& \le \frac{C_{1}}{\theta_{2}}\bigg(\frac{1}{\sqrt{\epsilon}}\bigg)+C_{2}(\Delta t)^{2}\le C\bigg(\frac{N^{-1}}{\sqrt{\epsilon}}+(\Delta t)^2\bigg).
	\end{align*}
		Choosing the barrier function for the layer component in $(x_{i},t_{j+1})\in(d-\tau_{2},d)\times(0, T]$ as
	$$\psi^{j+1}(x_{i})=C_{1}\bigg(\frac{(x_{i}-(d-\tau_{2}))N^{-1}}{\sqrt{\epsilon} \ln N}+(\Delta t)^{2}\bigg)\pm(W_{r}^{-(j+1)}-w_{r}^{-(j+1)})(x_{i})$$
	We can choose sufficiently large $C_{1}$ so that $\mathcal{L}^{N}\psi^{j+1}(x_{i})\le0, \forall~ (x_{i},t_{j+1})\in(d-\tau_{2},d)\times(0, T]$. Also $\psi^{j+1}(x_{\frac{3N}{8}})\ge0$ and $\psi^{j+1}(x_{\frac{N}{2}})\ge0$. Hence, by the discrete minimum principle \cite{ER1}, we can obtain the following bounds:
	\begin{align*}
	|(W_{r}^{-(j+1)}-w_{r}^{-(j+1)})(x_{i})|&\le C_{1} \bigg(\frac{(x_{i}-(d-\tau_{2}))N^{-1}}{\sqrt{\epsilon}\ln N}+(\Delta t)^{2}\bigg)\\& \le C_{1}\bigg(\frac{\tau_{2}N^{-1}}{\sqrt{\epsilon}\ln N}+(\Delta t)^{2}\bigg)
	\end{align*}
	\begin{equation}
	\label{WWWW}
	|(W_{r}^{-(j+1)}-w_{r}^{-(j+1)})(x_{i})|\le C(N^{-1}+(\Delta t)^{2}),~(x_{i},t_{j+1})\in(d-\tau_{2},d)\times(0, T].
	\end{equation}
Hence, by combining the equation\eqref{wwww} and \eqref{WWWW}, we have bounds for the right layer component
	\begin{equation}
	\label{wr-}
	\|W_{r}^{-}-w_{r}^{-}\|_{\Omega^{-}}\le C(N^{-1}+(\Delta t)^{2}).
	\end{equation}
\end{proof}
\begin{lemma}
	\label{L4}
		Let $\sqrt{\alpha}\mu\le \sqrt{\rho \epsilon}$ and $W_{r}^{+(j+1)}(x_{i}), w_{r}^{+}(x,t)$ are solution of the problem \eqref{RW+} and \eqref{rw} respectively. The right singular component of the truncation error satisfies the following estimate in $\Omega^{N+}$
	$$\|W_{r}^{+}-w_{r}^{+}\|_{\Omega^{N+}}\le C(N^{-1}+(\Delta t)^{2})$$
\end{lemma}	
\begin{proof}
	We first calculate the truncation error in the outer region $(d,1-\tau_{4}]\times(0, T]$:	
For $\frac{N}{2}< i\le \frac{7N}{8}$, the left layer component satisfies the following bound given in the theorem \eqref{wlb}:
	\begin{equation}
	\label{wr+}
	|w_{r}^{+(j+1)}(x_{i})|\le Ce^{-\theta_{1}(1-x_{i})}\le Ce^{-\theta_{1}\tau_{4}}\le CN^{-4},~(x_{i},t_{j+1})\in(d,1-\tau_{4}]\times(0,T].
	\end{equation}
	Also $W^{+(j+1)}_{r}(x_{i})$ is a increasing function in $(d,1-\tau_{4}]\times(0, T]$, so
	\begin{align*}
	|W^{+(j+1)}_{r}(x_{i})|&\le|W_{r}^{+(j+1)}(x_{\frac{7N}{8}})|= \prod_{n=\frac{7N}{8}+1}^{N}(1+\theta_{1}h_{n})^{-1}. 
	\end{align*}
	Applying the similar arguments as in Lemma \eqref{L1}, we get
	\begin{align}
	\label{wrr}
	|W_{r}^{+(j+1)}(x_{i})|&\le CN^{-4}, ~~\forall(x_{i},t_{j+1})\in(d,1-\tau_{4}]\times(0, T].
	\end{align}
	Hence for all $(x_{i},t_{j+1})\in (1,d-\tau_{2}]\times(0, T]$ from equation (\ref{wr+}) and (\ref{wrr}), we get
	\begin{equation}
	\label{wwwww}
	|(W_{r}^{+(j+1)}-w_{r}^{+(j+1)})(x_{i})|\le |W_{r}^{+(j+1)}(x_{i})|+|w_{r}^{+(j+1)}(x_{i})|\le CN^{-4}.
	\end{equation}
	
	We use truncation error analysis to find error in the inner region $(1-\tau_{4},1)\times(0, T] $. Using the derivative bounds for the left layer component $w_{r}^{+}$ from theorem \eqref{wwlbd}, we obtain
	\begin{align*}
	|\mathcal{L}^{N}(W^{+(j+1)}_{r}-w^{+(j+1)}_{r})(x_{i})|
	&\le \epsilon \bigg|\bigg(\delta^{2}_{x}-\frac{d^{2}}{dx^{2}}\bigg) w^{+(j+1)}_{r}(x_{i})\bigg|+\mu |a^{j+\frac{1}{2}}(x_{i})|\bigg|\bigg(D^{+}_{x}-\frac{d}{dx}\bigg)w^{+(j+1)}_{r}(x_{i})\bigg|\\&+\bigg|D_{t}^{-}w_{r}^{+(j+1)}(x_{i})-\frac{\delta}{\delta t}w_{r}^{+(j+\frac{1}{2})}(x_{i})\bigg|\\
	&\le C_{1} \max_{\frac{N}{2}\le i\le N} h_{i}(\epsilon\|w^{-}_{xxx}\|_{\Omega^{+}}+\mu\|w^{-}_{xx}\|_{\Omega^{+}})+C_{2}(\Delta t)^{2})\\
	& \le \frac{C_{1}}{\theta_{1}}\bigg(\frac{1}{\sqrt{\epsilon}}\bigg)+C_{2}(\Delta t)^{2}\le C\bigg(\frac{N^{-1}}{\sqrt{\epsilon}}+(\Delta t)^2\bigg).
	\end{align*}
	Choosing the barrier function in the domain $(1-\tau_{4},1)\times(0, T] $ for the layer component as
	$$\psi^{j+1}(x_{i})=C_{1}\bigg(\frac{(1-x_{i})N^{-1}}{\sqrt{\epsilon} \ln N}+(\Delta t)^{2}\bigg)\pm(W_{r}^{+(j+1)}-w_{r}^{+(j+1)})(x_{i}),$$
	For sufficiently large $C_{1}$ $\mathcal{L}^{N}\psi^{j+1}(x_{i})\le0,\forall~(x_{i},_{j+1})\in(1-\tau_{4},1)\times(0, T] $. Also $\psi^{j+1}(x_{\frac{7N}{8}})\ge0$ and $\psi^{j+1}(x_{N})\ge0$. Hence, by the discrete minimum principle \cite{ER1}, we can obtain the following bounds:
	\begin{align*}
	|(W_{r}^{+(j+1)}-w_{r}^{+(j+1)})(x_{i})|&\le C_{1} \bigg(\frac{(1-x_{i})N^{-1}}{\sqrt{\epsilon} \ln N}+(\Delta t)^{2}\bigg)\\& \le C_{1}\bigg(\frac{\tau_{4}N^{-1}}{\sqrt{\epsilon} \ln N}+(\Delta t)^{2}\bigg)
	\end{align*}
	\begin{equation}
	\label{WWWWW}
	|(W_{r}^{+(j+1)}-w_{r}^{+(j+1)})(x_{i})|\le C(N^{-1}+(\Delta t)^{2}),~(x_{i},t_{j+1})(1-\tau_{4},1)\times(0, T].
	\end{equation}
	Hence, by combining the equation\eqref{wwwww} and \eqref{WWWWW}, we have bounds for the left layer component
	\begin{equation}
	\label{w-}
	\|W_{r}^{+}-w_{r}^{+}\|_{\Omega^{N+}}\le C(N^{-1}+(\Delta t)^{2}).
	\end{equation}
\end{proof}
\begin{theorem}
		The error $e^{j+1}(x_{\frac{N}{2}})$ estimated at the point of discontinuity $(x_{\frac{N}{2}},t_{j+1})=(d,t_{j+1}), 0\le j\le M-1$ satisfies the following estimate for $\sqrt{\alpha}\mu \le \sqrt{\rho\epsilon}$:
		$$|(D^{+}-D^{-})(U^{j+1}(x_{\frac{N}{2}})-u^{j+1}(x_{\frac{N}{2}}))|\le C\frac{\bar{h}}{\epsilon}$$
		where $\bar{h}=\max\{h_{\frac{N}{2}},h_{\frac{N}{2}+1}\}$.
\end{theorem}
\begin{proof}
	Consider 
	$$|(D^{+}-D^{-})(U^{j+1}(x_{\frac{N}{2}})-u^{j+1}(x_{\frac{N}{2}}))|\le|(D^{+}-D^{-})u^{j+1}(x_{\frac{N}{2}}))|,$$
	since $|(D^{+}-D^{-})U^{j+1}(x_{\frac{N}{2}}))|=0.$\\
	Now
	\begin{align*}
	|(D^{+}-D^{-})(U^{j+1}(x_{\frac{N}{2}})-u^{j+1}(x_{\frac{N}{2}}))|&\le\bigg|\bigg(\frac{d}{dx}-D^{+}\bigg)u^{j+1}(d)\bigg|+\bigg|\bigg(\frac{d}{dx}-D^{-}\bigg)u^{j+1}(d)\bigg|\\&\le C_{1}h_{\frac{N}{2}+1}|u_{xx}|_{\Omega^{+}}+C_{2}h_{\frac{N}{2}}|u_{xx}|_{\Omega^{-}}\\&\le C\bar{h}|u_{xx}|_{\Omega^{+}\cup{\Omega^{-}}}\\&\le
\frac{C\bar{h}}{\epsilon}~~(\text{from theorem} \eqref{wwlbd}).
	\end{align*}
\end{proof}
\begin{theorem}
	\label{main}
	Let us assume $\sqrt{\epsilon}<N^{-1}$ and $\sqrt{\alpha}\mu \le \sqrt{\rho\epsilon}$. Let $u(x,t)$ and $U^{j+1}(x_{i})$ be the solutions of the continuous and discrete problems \eqref{twoparaparabolic} and \eqref{Full-dis} respectively, then,
	$$\|U-u\|_{\Omega}\le
	C(N^{-1}+\Delta t^{2})$$
	where C is a constant independent of $\epsilon, \mu$ and discretization parameter $N,M$. 
\end{theorem}
\begin{proof}
	Combining the lemmas \eqref{L1},\eqref{L2},\eqref{L3} and \eqref{L4}, we obtain the following bound for $(x_{i},t_{j+1})\ne(d,t_{j+1})$
	\begin{equation}
	\label{U}
	|(U-u)(x_{i},t_{j+1})|\le C(N^{-1}+\Delta t^{2}),~~\forall~(x_{i},t_{j+1})\in \Omega^{N-}\cup\Omega^{N+}.
	\end{equation}
To obtain error at the point of discontinuity $(x_{i},t_{j+1})=(x_{\frac{N}{2}},t_{j+1})$ for the first case $\sqrt{\alpha}\mu \le \sqrt{\rho\epsilon}$, we consider the discrete barrier function $\phi_{1}^{j+1}(x_{i})=\psi_{1}^{j+1}(x_{i})\pm e^{j+1}(x_{i})$ in the interval $(d-\tau_{2},d+\tau_{3})$ where
$$\psi_{1}^{j+1}(x_{i})=C_{1}(\Delta  t)^{2}+\frac{C_{2}\tau}{\epsilon N (\log N)^{2}}\left\{
\begin{array}{ll}
\displaystyle x_{i}-(d-\tau_{2}), & \hbox{ $(x_{i},t_{j+1})\in \Omega^{N,M}\cap(d-\tau_{2}, d] $, }\\
d+\tau_{3}-x_{i}, & \hbox{ $(x_{i},t_{j+1})\in\Omega^{N,M}\cap[d, d+\tau_{3})$.}
\end{array}
\right. $$
where $\tau=\max\{\tau_{2},\tau_{3}\}$. It could be seen that $\phi_{1}^{j+1}(d-\tau_{2})$ and $\phi_{1}^{j+1}(d+\tau_{3})$ are non negative and
$\mathcal{L}^{N}\phi_{1}^{j+1}(x_{i})\le 0,~~(x_{i},t_{j+1})\in\Omega^{N,M},~~|(D^{+}-D^{-})\phi_{1}^{j+1}(x_{\frac{N}{2}}))|\le0.$\\
Hence, by applying the discrete minimum principle, we get, $\phi_{1}^{j+1}(x_{i})\ge0.$\\
Therefore, for $(x_{i},t_{j+1})\in (d-\tau_{2},d+\tau_{3})$
\begin{equation}
\label{Yd}
|e^{j+1}({x_{i}})|=|(U-u)(x_{i},t_{j+1})|\le C_{1}(\Delta t)^{2}+\frac{C_{2}\tau^{2}}{\epsilon N(\log N)^{2}}\le C(N^{-1}+(\Delta t)^{2}).
\end{equation}
Combining the equation \eqref{U} and \eqref{Yd}, we get the required result.
\end{proof}


\section{Numerical Examples}

In this section, we examine two test problems with discontinuous convection coefficients and discontinuous source terms to validate the proposed method.
\begin{Example}\label{ex-a}
	$$\epsilon u_{xx}+\mu a(x,t)u_{x}-b(x,t)u-u_{t}=f(x,t)~~ ~~(x,t)\in(0,.5)\cup(0.5,1)\times(0,1],$$
	$$	u(0,t)=u(1,t)=u(x,0)=0,$$\nonumber\
	with \\
	$$a(x,t)=\left\{
	\begin{array}{ll}
	\displaystyle -(1+x(1-x)), & \hbox{ $0\le x\le 0.5,t\in(0,1]$, }\\
	1+x(1-x), & \hbox{ $0.5<x\le1,t\in(0,1]$,}
	\end{array}
	\right.$$
	and
	$$
	f(x,t)=\left\{
	\begin{array}{ll}
	\displaystyle -2(1+x^2)t, & \hbox{ $0\le x\le 0.5,t\in(0,1]$, }\\
	2(1+x^2)t, & \hbox{ $0.5<x\le1,t\in(0,1]$.}
	\end{array}
	\right.$$\\
	$b(x,t)=1+exp(x),~\text{and} ~c(x,t)=1$.
\end{Example}
\begin{Example}\label{ex-b}
	$$\epsilon u_{xx}+\mu a(x,t)u_{x}-b(x,t)u-u_{t}=f(x,t)~~ ~~(x,t)\in(0,.5)\cup(0.5,1)\times(0,1],$$
	$$	u(0,t)=u(1,t)=u(x,0)=0,$$\nonumber\
	with \\
	$$a(x,t)=\left\{
	\begin{array}{ll}
	\displaystyle -(1+x(1-x)), & \hbox{ $0\le x\le 0.5,t\in(0,1]$, }\\
	1+x(1-x), & \hbox{ $0.5<x\le1,t\in(0,1]$,}
	\end{array}
	\right. $$	
	and
	$$
	f(x,t)=\left\{
	\begin{array}{ll}
	\displaystyle -2(1+x^2)t, & \hbox{ $0\le x\le 0.5,t\in(0,1]$, }\\
	3(1+x^2)t, & \hbox{ $0.5<x\le1,t\in(0,1]$.}
	\end{array}
	\right.$$\\
	$b(x,t)=1+exp(x),~\text{and} ~c(x,t)=1$.
\end{Example}
As the exact solutions of eg. \ref{ex-a} and eg. \ref{ex-b} is unknown, to calculate the maximum point-wise error and the rate of convergence, we use the double mesh principle \cite{DV,D}. The double mesh difference is defined by
$$E_{\epsilon,\mu}^{N,M}= \max\limits_{j}\bigg(\max\limits_{i}|U_{2i,2j}^{2N,2M}-U_{i,j}^{N,M}|\bigg)$$
where $U_{i}^{N,M}$ and $U_{2i}^{2N,2M}$ are the solutions on the mesh $\bar{\Omega}^{N,M}$ and $\bar{\Omega}^{2N,2M}$ respectively. 
The order of convergence is given by
$$R_{\epsilon,\mu}^{N,M}=\log_{2}\bigg(\frac{E_{\epsilon,\mu}^{N,M}}{E_{\epsilon,\mu}^{2N,2M}}\bigg).$$
In the numerical examples, we have taken $N=M$.

The Figure \ref{fig21} represents the surface plot of numerical solution and error graph for $\epsilon=10^{-12}$, $\mu=10^{-8}$ and $N=64$. We note that the maximum point-wise error is at the point of discontinuity. 
In Table \ref{eg11}, we present the maximum point-wise error $E_{\epsilon,\mu}^{N,M}$ and approximate orders of convergence $R_{\epsilon,\mu}^{N,M}$ for Example \ref{ex-a} for various values of $\mu$ when $\epsilon=10^{-8}$. 
\begin{table}[ht!]
	\centering
	\caption{Maximum point-wise error $E_{\epsilon,\mu}^{N,M}$ and approximate orders of convergence $R_{\epsilon,\mu}^{N,M}$ for Example \ref{ex-a} when $\epsilon=10^{-8}$}
	\label{eg11}
	\begin{tabular}{ |c|c|c|c|c| } 
		\hline
		\multirow{2}{*}{$\mu$ }&\multicolumn{4}{c|}{Number of mesh points N} \\ 
	\cline{2-5}
		  & 64 & 128 & 256 & 512 \\ 
		\hline
		$10^{-6}$ & 0.039036&	0.019471&	0.009595&	0.004722  \\ 
		\hline
		Order & 1.0035&	1.0210&	1.0229& \\
		\hline
		$10^{-7}$ &0.039041&	0.019476&	0.009597&	0.004723 \\
		\hline
		Order &1.0033	&1.0211&	1.0230&	 \\
		\hline
		$10^{-8}$ & 0.039042&	0.019477&	0.009597&	0.004722 \\ 
		\hline
		Order  & 0.7480	&0.8690	&0.9254 &	\\
		\hline
		$10^{-9}$ & 0.039042&	0.019477&	0.009597&	0.004722\\ 
		\hline
		Order  & 1.0032&	1.0211&	1.0230&\\
		\hline
		$10^{-10}$ & 0.039042&	0.019477&	0.009591&	0.004722\\ 
		\hline
		Order  &1.0032&	1.0211&	1.0230&\\
		\hline
		$10^{-11}$ &0.039042&	0.019477&	0.009597&	0.004722\\ 
		\hline
		Order  &1.0032&	1.0211&	1.02303&\\
		\hline
		$10^{-12}$ &0.039043&	0.019487&	0.009599&	0.004732\\ 
		\hline
		Order  &1.0025&	1.0215&	1.0203&\\
		\hline
		$10^{-13}$ &0.039041&	0.019492&	0.009560&	0.004751\\ 
		\hline
		Order  &1.0024&	1.0274&	1.0083&\\
		\hline
		$10^{-14}$ &0.039042&	0.019477&	0.009597&	0.004722\\ 
		\hline
		Order  &1.0024&	1.0274&	1.0083&\\
		\hline
		$10^{-15}$ &0.039042&	0.019477&	0.009597&	0.004722\\ 
		\hline
		Order  &1.0024&	1.0274&	1.0083&\\
		\hline
		$10^{-16}$&0.039042&	0.019477&	0.009597&	0.004722\\ 
		\hline
		Order  &1.0024&	1.0274&	1.0083&\\
		\hline
	\end{tabular}
\end{table}
\begin{figure}[ht!]\small
	\begin{subfigure}{0.3\textwidth}
		\centering
		{\includegraphics[width=3in,height=2in]{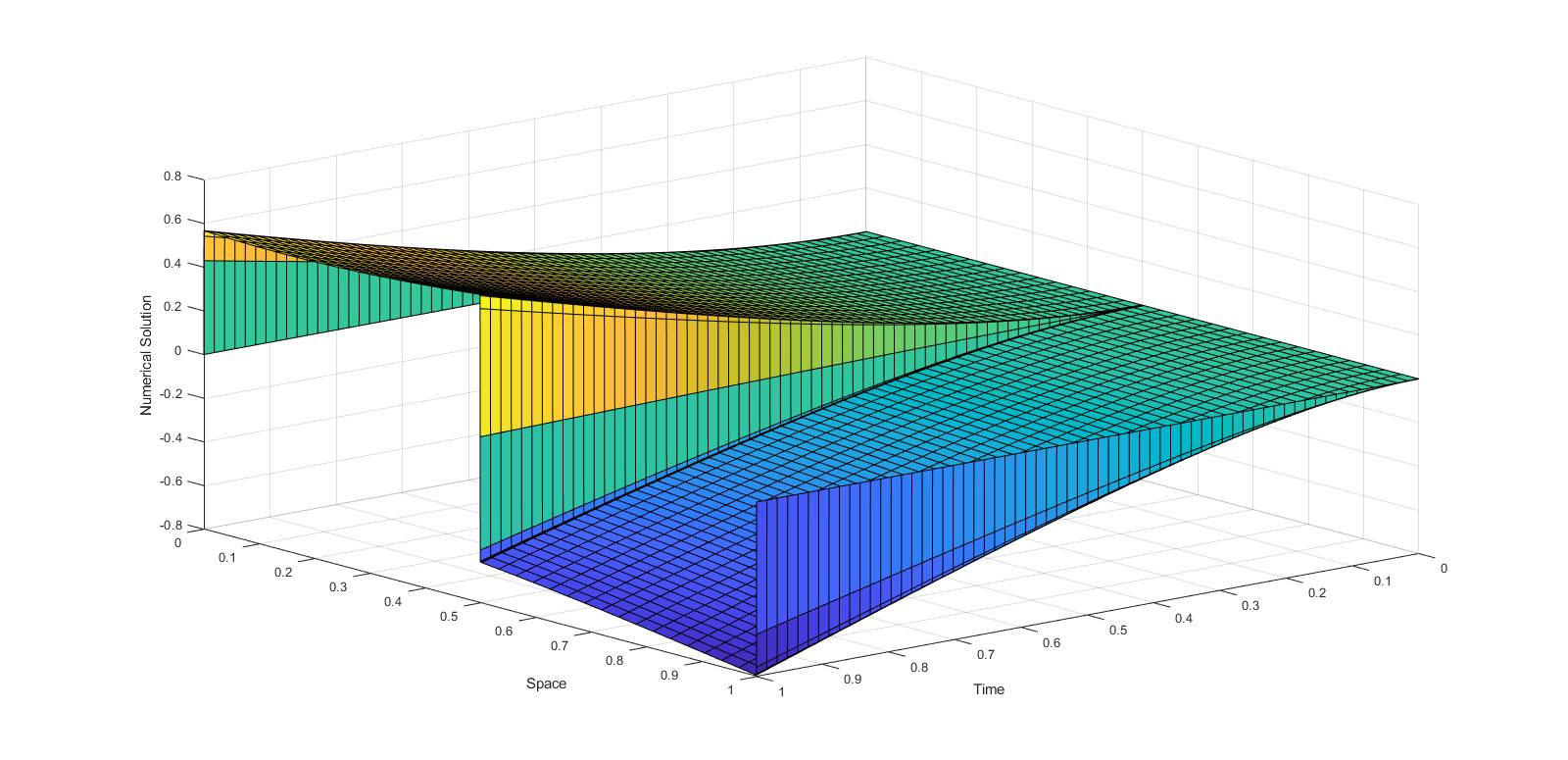}}
		\caption{Numerical Solution}
	\end{subfigure}
	\hfill
	\begin{subfigure}{0.5\textwidth}
		\centering
		{\includegraphics[width=3in,height=2.in]{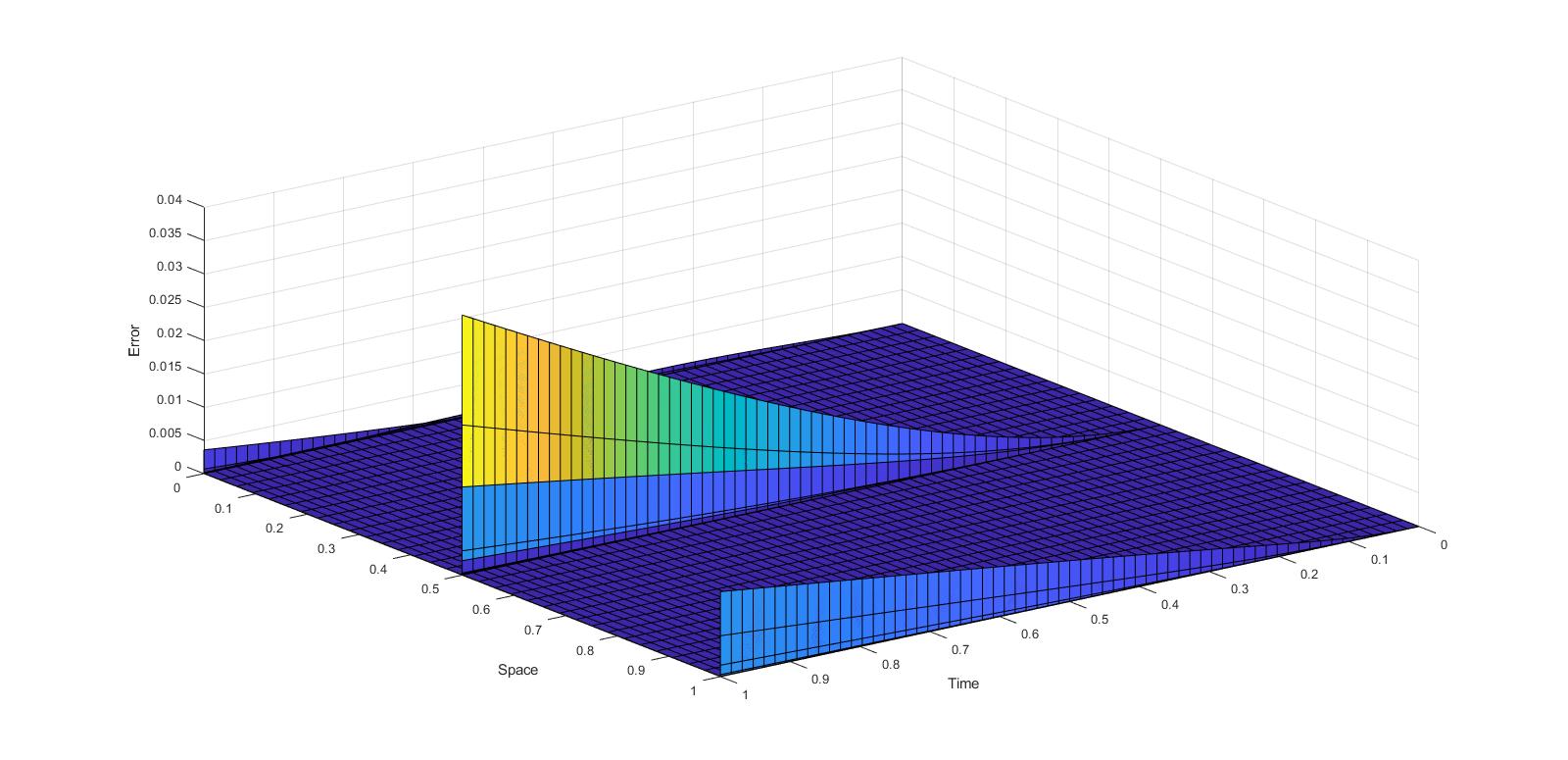}}
		\caption{Error$=|Y^{2N,2M}(x_{i})-Y^{N,M}(x_{i})|$}	
	\end{subfigure}
	\caption{(a) and (b): Plots of numerical solution and errors for $\epsilon=10^{-12}, \mu=10^{-8}$} when $N=64$ for Example \ref{ex-a}.	
	\label{fig21}
\end{figure}\\

The Figure \ref{fig22} represents the surface plots of numerical solution and error graph for $\epsilon=10^{-8}$, $\mu=10^{-8}$ and $N=128$.
In Table \ref{eg12}, we present the maximum point-wise error and  $E_{\epsilon,\mu}^{N,M}$ and approximate orders of convergence $R_{\epsilon,\mu}^{N,M}$ for Example \ref{ex-b} for various values of $\mu$ when $\epsilon=10^{-12}$. In both examples, we note that the overall order of convergence is one as $N=M$.

\begin{table}[ht!]
	\centering
	\caption{Maximum point-wise error $E_{\epsilon,\mu}^{N,M}$ and approximate orders of convergence $R_{\epsilon,\mu}^{N,M}$ for Example \ref{ex-b} when $\epsilon=10^{-12}$}
	\label{eg12}
	\begin{tabular}{ |c|c|c|c|c| } 
		\hline
		\multirow{2}{*}{$\mu$ }&\multicolumn{4}{c|}{Number of mesh points N} \\ 
			\cline{2-5}
  & 64 & 128 & 256 & 512 \\ 
		\hline
		$10^{-7}$ & 0.048709&	0.024266&	0.011958&	0.005888  \\ 
		\hline
		Order &1.00522&	1.02091&	1.02206& \\
		\hline
		$10^{-8}$ &0.048775&	0.024330&	0.011989&	0.005900 \\
		\hline
		Order &1.00337&	1.02100&	1.02291&	 \\
		\hline
		$10^{-9}$ & 0.048782&	0.024337&	0.011992&	0.005901 \\ 
		\hline
		Order  &1.00318&	1.02101&	1.02301&\\
		\hline
		$10^{-10}$ &0.048782&	0.024337&	0.011992&	0.005901\\ 
		\hline
		Order  & 1.00316&	1.02101&	1.02301&\\
		\hline
		$10^{-11}$ & 0.048782&	0.024337&	0.011993&	0.005901\\ 
		\hline
		Order  &1.00316&	1.02101&	1.02302&\\
		\hline
		$10^{-12}$ &0.048782&	0.024338&	0.011993&	0.005901\\ 
		\hline
		Order  &1.00316&	1.02101&	1.02302&\\
		\hline
		$10^{-13}$ &0.048782&	0.024338&	0.011993&	0.005901\\ 
		\hline
		Order  &1.00316&	1.02101&	1.02302&\\
		\hline
		$10^{-14}$ &0.048783&	0.024338	&0.01199301&	0.0059016\\ 
		\hline
		Order  &1.00316&	1.02101&	1.02302&\\
		\hline
		$10^{-15}$&0.048783&	0.024338	&0.01199301&	0.0059016\\ 
		\hline
		Order  &1.00316&	1.02101&	1.02302&\\
		\hline
		$10^{-16}$ &0.048783&	0.024338	&0.01199301&	0.0059016\\ 
		\hline
		Order  &1.00316&	1.02101&	1.02302&\\
		\hline
		$10^{-17}$ &0.048783&	0.024338	&0.01199301&	0.0059016\\ 
		\hline
		Order  &1.00316&	1.02101&	1.02302&\\
		\hline
	\end{tabular}
\end{table}
\begin{figure}[ht!]\small
	\begin{subfigure}{0.3\textwidth}
		\centering
		{\includegraphics[width=3in,height=2in]{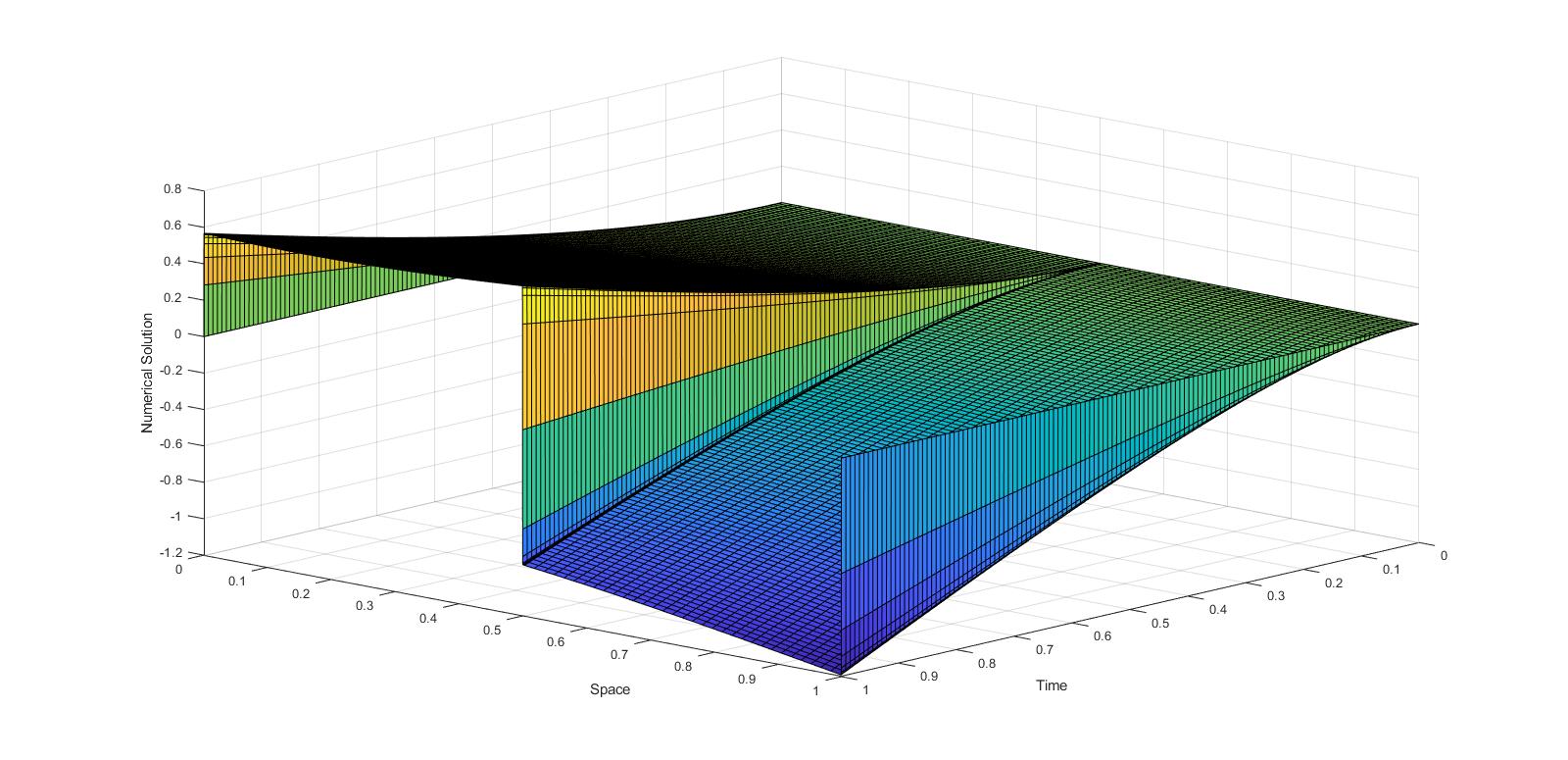}}
		\caption{Numerical Solution}
	\end{subfigure}
	\hfill
	\begin{subfigure}{0.5\textwidth}
		\centering
		{\includegraphics[width=3in,height=2.in]{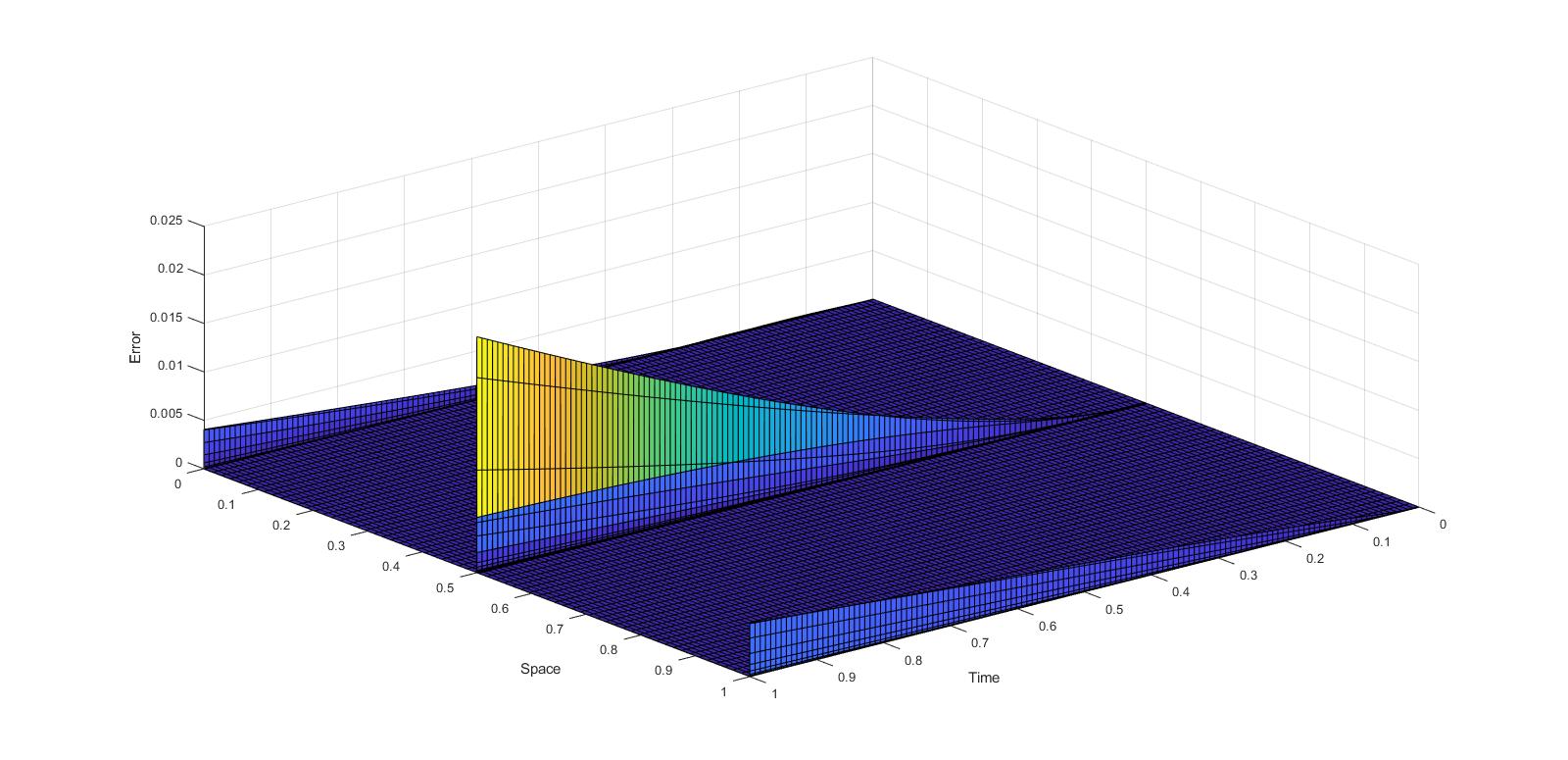}}
		\caption{Error$=|Y^{2N,2M}(x_{i})-Y^{N,M}(x_{i})|$}	
	\end{subfigure}
	\caption{(a) and (b): Plots of numerical solution and errors for $\epsilon=10^{-8}, \mu=10^{-8}$} when $N=128$ for Example \ref{ex-b}.	
	\label{fig22}
\end{figure}

\section{Conclusions}
Here, we have considered a two-parameter singularly perturbed parabolic problem with a discontinuous convection co-efficient and source term. The solution to this problem exhibits boundary layers near the boundaries and internal layers near the point of discontinuity due to the discontinuity of the convection coefficient and the source term. We have proposed a parameter uniform numerical scheme for the case  $\sqrt{\alpha} \mu\le\sqrt{\rho\epsilon}$. In the temporal direction, we employed the Crank-Nicolson method on a uniform mesh, and in the spatial direction, we used an upwind finite difference scheme on a Shishkin-Bakhvalov mesh. At the point of discontinuity, we used a three-point formula. The use of Crank-Nicolson method and Shishkin-Bakhvalov mesh helped to obtain second-order accuracy in time and first-order accuracy in space; unlike the Shishkin mesh, where a logarithmic factor deteriorates the order of convergence. The theoretical analysis was supported by the numerical results. 


\section*{Declaration of interests}
\mbox{\ooalign{$\checkmark$\cr\hidewidth$\square$\hidewidth\cr}} The authors declare that they have no known competing financial interests or personal relationships that could have appeared to influence the work reported in this paper.

\end{document}